\documentclass[10pt,a4paper]{article}
\usepackage[T1]{fontenc}
\usepackage{fancyhdr}
\usepackage{amsmath,amsthm,amsfonts,amssymb} 
\usepackage{epic,eepic} 
\usepackage[dvips]{graphicx,color} 
\usepackage{authblk}
\usepackage{makeidx, enumerate,hyperref}
\usepackage[active]{srcltx} 
\usepackage[style=numeric,hyperref,backend=bibtex]{biblatex}
\bibliography{bibbase}
\theoremstyle{plain}    
\newtheorem{theorem}{Theorem}[section]
\newtheorem{proposition}[theorem]{Proposition}
\newtheorem{lemma}[theorem]{Lemma}
\newtheorem{corollary}[theorem]{Corollary}
\theoremstyle{definition}
\newtheorem{definition}[theorem]{Definition} 
\theoremstyle{remark}
\newtheorem{remark}[theorem]{Remark}
\newtheorem{example}[theorem]{Example}
\numberwithin{equation}{section} %% Comment out for sequentially-numbered
\numberwithin{figure}{section} %% Comment out for sequentially-numbered
\numberwithin{theorem}{section}
\begin{document}
\begin{center}
{\LARGE{\bf Sensitivity relations for the Mayer problem with differential inclusions} }
% TITLE
\
\vskip0.5\baselineskip{\bf P. Cannarsa$^{1}$, H. Frankowska$^{2,3,4}$, T. Scarinci$^{1,3,4}$ }
% AFFILIATION
\vskip0.5\baselineskip{\em$^{1}$ Dipartimento di Matematica,
Universit\`a di Roma 'Tor Vergata', 
Via della Ricerca Scientifica 1, 00133 Roma, Italy \\
$^{2}$CNRS, UMR 7586,  F-75005, Paris, France \\
$^{3}$Sorbonne Universit\'es, UPMC Univ Paris 06, UMR 7586, IMJ-PRG, case 247,  4 place Jussieu,  F-75005, Paris, France\\
$^4$Univ Paris-Diderot, Sorbonne Paris Cit\'e, UMR 7586, F-75013, Paris, France}
\end{center}
\section*{Abstract}
In optimal control,
% literature (see, for instance, \cite{MR2662630}), 
 sensitivity relations are usually understood as inclusions that identify the pair formed by the dual arc and the Hamiltonian, evaluated along the associated minimizing trajectory, as a suitable generalized gradient of the value function.
In this paper, sensitivity relations are  obtained for the Mayer problem associated with the differential inclusion $\dot x\in F(x)$ and applied to derive  optimality conditions.
% Here  $F$ is a Marchaud multifunction such that the Hamiltonian $H(x,p)=\max_{v\in F(x)}\langle v,p\rangle$ is semiconvex in $x$ and differentiable in $p$ for $p\neq 0$. Such relations ensure that the dual arc of a given optimal trajectory belongs to the superdiffential of the value function along the trajectory and can be  
Our first  application concerns the maximum principle and consists in showing that a dual arc can be constructed for {\em every} element of the superdifferential of the final cost. 
As our second application, with every nonzero limiting gradient of the value function at some point $(t,x)$ we  associate a family of optimal trajectories at $(t,x)$ with the property that families corresponding to distinct limiting gradients have {\em empty} intersection.
\section{Introduction} 
%The above concept is easy to explain for the Mayer problem. 
Given a complete separable metric space $U$ and a vector field $f:\mathbb R^n\times U\to\mathbb R^n$,  smooth with respect to $x$, for any point $(t_0,x_0)\in [0,T]\times\mathbb R^n$ and  Lebesgue measurable map $u:[t_0,T]\to U$ let us denote by $x(\cdot;t_0,x_0,u)$ the solution of the Cauchy problem
\begin{equation}\label{ControlCase}
\begin{cases}
\hspace{.cm}
\dot x(t)=f(x(t),u(t))
&t\in [t_0,T]\;\text{a.e.}
\\
\hspace{.cm}
x(t_0)=x_0,
\end{cases}
\end{equation}
that we suppose to exist on the whole interval $[t_0,T]$. Then, given a smooth function $\phi:\mathbb R^n\to\mathbb R$, we are interested in minimizing
the final cost $\phi(x(T;t_0,x_0,u))$ over all controls $u$. 

In the Dynamic Programming approach to such a problem, one seeks to characterize the value function $V$, that is,
\begin{equation}\label{eq:introV}
V(t_0,x_0)= \inf_{u(\cdot)} \phi(x(T;t_0,x_0,u(\cdot)))\qquad(t_0,x_0)\in [0,T]\times\mathbb R^n,
\end{equation}
as the unique solution, in a suitable sense, of the Hamilton-Jacobi equation
\begin{equation*}
\label{HJBeq}
\begin{cases}
\hspace{.cm}
-\partial_t v(t,x) + H(x,-v_x (t,x) )=0&
 \mbox{ in } (0,T)\times \mathbb{R}^n
\\
\hspace{.cm}
\hspace{.3cm}v(T,x)=\phi(x)
&
x\in \mathbb{R}^n,
\end{cases}
\end{equation*}
where the Hamiltonian $H$ is defined as
\begin{equation*}
H(x,p)=\sup_{u\in U}\langle p,f(x,u)\rangle\qquad(x,p)\in \mathbb R^n\times \mathbb R^n.
\end{equation*}
Now, the classical method of characteristics ensures that, as long as $V$ is smooth, along any solution of the system of ODEs
\begin{equation}\label{eq:char}
\begin{cases}
\hspace{.cm}
\hspace{.3cm}\dot x(t)=\nabla_p H(x(t),p(t))
,&
x(T)=z
\\
\hspace{.cm}
-\dot p(t)=\nabla_xH(x(t),p(t))
,&
p(T)=-\nabla\phi(z)
\end{cases}
\quad t\in[0,T],
\end{equation}
%with terminal conditions
%\begin{equation*}
%\begin{cases}
%\hspace{.cm}
%x(T)=z
%\\
%\hspace{.cm}
%p(T)=-\nabla\phi(z)
%\end{cases}
%\quad z\in\mathbb R^n
%\end{equation*}
the gradient of $V$ satisfies
\begin{equation}\label{eq:introSR1}
(H(x(t),p(t)),-p(t))=\nabla V(t,x(t)), \quad \forall\ t\in[0,T].
\end{equation}
It is well known that the characteristic system \eqref{eq:char} is also a set of necessary optimality conditions for any optimal solution $x(\cdot)$ of the Mayer problem \eqref{eq:introV}. The mapping $p(\cdot)$ is called a dual arc. Observe that $\nabla V (t,x)$ allows to ``measure'' sensitivity of the optimal cost with respect to $(t,x)$. For this reason, \eqref{eq:introSR1} is called a \emph{sensitivity relation} for problem \eqref{eq:introV}.
Obviously, the above calculation is just formal because, in general,   $V$ cannot be expected to be smooth. On the other hand, relation \eqref{eq:introSR1} is important for deriving sufficient optimality conditions, as we recall in Section 2. This fact motivates interest  in generalized sensitivity relations for nonsmooth value functions.

To the best of our knowledge, the first ``nonsmooth result'' in the above direction was obtained by  Clarke and Vinter in \cite{Clarke:1987:RMP:35498.35509} for the Bolza  problem, where,
given an optimal trajectory $x_0(\cdot)$, an associated dual arc $p(\cdot)$ is constructed to satisfy the partial sensitivity relation 
\begin{equation} \label{IntroSensParti}
-p(t)\in \partial_x V(t, x_0(t)) \textit{ a.e. } t\in[t_0,T].
\end{equation}
Here, $\partial_x V$ denotes  Clarke's generalized   gradient of $V$ in the second variable. Subsequently, 
 for the same problem, Vinter~\cite{MR923279} proved the existence of a dual arc satisfying the full sensitivity relation
\begin{equation}\label{IntroSensTotal}
\left( H(x_0(t),p(t)), -p(t) \right)\in \partial V(t,x_0(t))\mbox{ for all } t\in [t_0,T],
\end{equation}
with $\partial V$ equal to  Clarke's generalized  gradient in $(t,x)$.

Full sensitivity relations were recognized as necessary and sufficient conditions for optimality in \cite{Cannarsa:1991:COT:120771.118946}, where the first two authors of this paper studied the Mayer problem for the parameterized control system \eqref{ControlCase}, with $f$  depending also on time. More precisely, replacing the Clarke generalized gradient with the Fr\'echet superdifferential $\partial^+ V$, the full sensitivity relation
\begin{equation}
\left( H(t,x_0(t),p(t)), -p(t) \right)\in \partial^+ V(t,x_0(t))\quad \mbox{ a.e. } t\in [t_0,T],
\end{equation}  
together with the maximum principle
\begin{equation}
\langle p(t),\dot{x}_0(t)\rangle = H(t,x_0(t),p(t))\quad \mbox{ a.e. } t\in [t_0,T],
\end{equation}
was shown to actually characterize optimal trajectories. A similar result was proved in \cite{MR894990}, under stronger regularity assumptions.
%  than in \cite{Cannarsa:1991:COT:120771.118946}.\\

Following the above papers, the analysis has been extended in several directions. For instance, in \cite{MR1780579},  sensitivity relations were adapted to the minimum time problem for the parameterized control system 
\begin{equation}
\dot{x}(t)=f(x(t),u(t)) \quad t \geq 0,
\end{equation}
taking the form of the inclusion
\begin{equation}\label{ForMinTime}
-p(t)\in \partial^+ T(x_0(t)) \mbox{ for all } t\in [0,T(x_0)),
\end{equation}
where $T(\cdot)$ denotes the minimum time function for a target $K$, and $x_0(\cdot)$ an optimal trajectory starting from $x_0$ which attains $K$ at time $T(x_0)$. 
%In this case, the major difficulty was to identify the right transversality condition that allows for the supergradient inclusion as suitable inner normal to the target. 
In \cite{MR3005035}, the above result has been  extended to nonparameterized systems described by  differential inclusions.
%\begin{equation}
%\dot{x}(t)\in F(x(t)).
%\end{equation} 
%It is worth nothing that, even for time optimal control problem, the full sensitivity relation as necessary and sufficient condition for optimality are obtained coupling \eqref{ForMinTime} with the constancy of the Hamiltonian, that is 
%\begin{equation}
%H(x(t),u(t))=1 \mbox{ a.e. } t\in (0,T(x)).
%\end{equation}
As for optimal control problems with state constraints, sensitivity relations were derived in
 \cite{BFV} and
 \cite{frankowska:hal-00800199}
% \cite{Bettiol:2010:SIC:1958083.1958102} 
% . This problem was further investigated in 
% for differential inclusions with pathwise state constraints.
using a suitable relaxation of the limiting subdifferential of the value function.

Deriving sensitivity relations in terms of the Fr\'echet and/or proximal superdifferential of the value function for the differential inclusion
\begin{equation}\label{May1}
\dot{x}(s)\in F(x(s))\quad \mbox{  a.e. } s \in [t_0,T],
\end{equation}
with the initial condition
\begin{equation}\label{May2}
x(t_0)=x_0,
\end{equation}
is far from straightforward, when  $F$ cannot be parameterized as
\begin{equation*}
F(x)=\{f(x,u)~:~u\in U\}
\end{equation*}
with $f$ smooth in $x$. The main goal of the present work is to obtain both partial and full sensitivity relations for the Mayer problem
\begin{equation}\label{Mayer}
\inf \phi(x(T)),
\end{equation}
the infimum being taken over all absolutely continuous arcs $x:[t_0,T]\rightarrow \mathbb{R}^n$ that satisfy \eqref{May1}-\eqref{May2}. The main assumptions we impose on the data,  expressed in terms of the Hamiltonian 
\begin{equation}\label{H}
H(x,p)= \sup_{v \in F(x)} \langle v,p\rangle \qquad (x,p)\in \mathbb{R}^n\times \mathbb{R}^n ,
\end{equation}  
require  $H(\cdot,p)$ to be semiconvex,  $H(x,\cdot)$ to be differentiable for $p\neq 0$, and $\nabla_pH(\cdot, p)$ locally Lipschitz continuous.  We refer the reader to \cite{MR2728465}, where this set of  hypotheses was  used to obtain the semiconcavity of the value function, for a detailed discussion of their role in lack of  a smooth parameterization of $F$.

For the Mayer problem \eqref{Mayer}, we shall derive sensitivity relations like \eqref{IntroSensParti} and \eqref{IntroSensTotal} for both the proximal and Fr\'echet superdifferential of the value function. More precisely, 
let  $\overline{x}:[t_0,T]\rightarrow \mathbb{R}^n$ be an optimal trajectory of  problem \eqref{Mayer} and let 
$\overline{p}:[t_0,T]\rightarrow \mathbb{R}^n$ be any arc satisfying 
\begin{equation}\label{intro:DualArc4}
\begin{cases}
\hspace{.cm}
-\dot{\overline p}(t)\in\partial_x^- H(\overline{x}(t),\overline p(t))
\\
\hspace{.3cm}
\dot{\overline{x}}(t)\in \partial_p H(\overline{x}(t),\overline p(t)) 
\end{cases}
\quad \mbox{ a.e. in } \left[ t_0, T \right],
\end{equation}
and 
\begin{equation}\label{intro:dualarcStarAAA}
-\overline{p}(T)\in \partial^{+,pr} \phi(\overline x(T)),
\end{equation}
where $\partial^{+,pr} \phi$ denotes the proximal superdifferential of $\phi$, and $\partial_x^- H$ denotes the Fr\'echet subdifferential of $H$ with respect to $x$. 
\footnote{We will see that solutions $(\overline{x}(\cdot),\overline{p}(\cdot)$) of \eqref{intro:DualArc4} are in the set of differentiability of the map $H(x,\cdot)$ when $p(T)\neq 0$. In that case, the second equation of \eqref{intro:DualArc4} becomes $\dot{\overline{x}}(t)=\nabla_p H (\overline{x}(t),\overline{p}(t))$ for all $t\in [t_0,T]$.}
Then we show that $\overline{p}(\cdot)$ satisfies the proximal partial sensitivity relation 
\begin{equation}\label{intro:PPSR}
-\overline{p}(t)\in \partial_x^{+,pr}V(t,\overline{x}(t))\mbox{ for all }t\in[t_0,T].
\end{equation}
Moreover, replacing $\partial^{+,pr} \phi(\overline x(T))$ by the Fr\'echet superdifferential $\partial^{+} \phi(\overline x(T))$ in \eqref{intro:dualarcStarAAA}, we derive the full sensitivity relation
\begin{equation}\label{intro:NuovaInclu}
\left( H(\overline{x}(t),\overline{p}(t)),-\overline{p}(t)\right)\in  \partial^+ V(t,\overline{x}(t)) \mbox{ for all } t\in (t_0,T).
\end{equation}
Thanks to \eqref{intro:NuovaInclu} we can recover, under suitable assumptions, the same set of necessary and sufficient conditions for optimality that appears in the context of smooth parameterized systems.

From a technical viewpoint, we note that the proof of \eqref{intro:PPSR} and \eqref{intro:NuovaInclu} is entirely different from the one which is used for parameterized control systems. Indeed, in the latter case, the conclusion is obtained appealing to the variational equation of \eqref{ControlCase}. In the present context, such a strategy is impossible to follow because $F$ admits no smooth parameterization, in general. As in \cite{MR2728465}, the role of the variational equation is here played by the maximum principle.

After obtaining sensitivity relations, we discuss two applications of \eqref{intro:PPSR} to the Mayer problem. Our first application is concerned with optimality conditions. Under our assumptions on $H$, the maximum principle in its available forms associates, with any optimal trajectory  $\overline{x}:[t_0,T]\rightarrow \mathbb{R}^n$ of  problem \eqref{Mayer}, a dual arc 
$\overline{p}:[t_0,T]\rightarrow \mathbb{R}^n$ such that $(\overline x,\overline p)$ satisfies  \eqref{intro:DualArc4}
and the transversality condition 
\begin{equation}\label{intro:dualarcStar4}
-\overline{p}(T)\in \partial \phi(\overline x(T)),
\end{equation}
see, for instance, \cite{MR709590}. Here, for $F$ locally strongly convex, we show how to construct multiple dual arcs $p(\cdot)$ satisfying the maximum principle
\begin{equation}\label{intro:MaxPrinc4}
H(\overline{x}(t),p(t))= \langle  p(t), \dot{\overline{x}}(t) \rangle\quad \mbox{ a.e. in } [t_0,T],
\end{equation}
by solving, for any $q \in \partial^{+,pr} \phi(\overline{x}(T))$, the terminal value problem
\begin{equation*}%\label{defP}
\left\{\begin{array}{l}
-\dot{p}(s) \in \partial_x^- H(\overline{x}(s),p(s))\quad \mbox{ a.e. in } s\in \left[ t_0, T \right], \\
-p(T)=q.
\end{array}\right.
\end{equation*}
  
Our second application aims to clarify the connection between the limiting gradients of $V$ at some point $(t,x)$, $\partial^*V(t,x)$, and the optimal trajectories at $(t,x)$. When the control system is parameterized as in \eqref{ControlCase}, such a connection is fairly well understood: one can show that any 
nonzero limiting gradient of $V$ at $(t,x)$  can be associated with an optimal trajectory  starting from $(t,x)$, and the map from  $\partial^*V(t,x)\setminus \{0\}$ into the family of optimal trajectories is one-to-one (see \cite[Theorem~7.3.10]{MR2041617}). In this paper, we use a suitable version
of \eqref{intro:PPSR} to prove an analogue of the above result (Theorem~\ref{Strano}) which takes inot account the lack of uniqueness for the initial value problem \eqref{intro:DualArc4}-\eqref{intro:dualarcStarAAA}.
%
%... comparison with the standart version. Discuss the geometric assumptions that allows to have our estimates. Recall R convexity. \\
%Furthrmore, 
%
%derivare le sensitivity relation é importante per derivare le condizioni sufficienti di ottimalita! ricordare l'articolo cannarsa frankowska.
%
%Another application of the sensitivity relations is given by the following classification of optimal trajectories. For each point of the ...\\

This paper is organized as follows. In Section $2$, we set our notation, introduce the main assumptions of the paper, and recall preliminary results from nonsmooth analysis and control theory. In Section $3$, sensitivity relations are derived in terms of the proximal and Fr\'echet superdifferentials. Finally, an application to the maximum principle is obtained in Section $4$, and a result connecting limiting gradients of $V$ with  optimal trajectories in Section $5$.

%Under minimal hypotheses that are in force in this paper, $H$ is upper semicontinuous in $(x,p)$, continuous and positively homogeneous of degree one in the $p$ variable. Note that this map can not be differentiable at $p =0$, except for trivial cases. There is an one-to-one relationship beetwen $H$ and $F$, that is the following
%\[
%v \in F(x) \Longleftrightarrow  \langle v , p \rangle \leq H(x,p),\ \forall p \in \mathbb{R}^n.
%\]
%Hence, assumptions that we will give in terms of the Hamiltonian will be intrinsic to the trajectories generated by $F$, without any infuence from particular parameterizations of the problem.

\section{Preliminaries}
\subsection{Notation}
Let us start by listing various basic notations and quickly reviewing some general facts for future use. Standard references are \cite{MR2041617,MR709590}. \\

We denote by $\mathbb{R}^+$ the set of strictly positive real numbers, by $|\cdot|$ the Euclidean norm in $\mathbb{R}^n$, and by $\langle \cdot,\cdot \rangle$ the inner product. $B(x,\epsilon)$ is the closed ball of radius $\epsilon > 0$ and center $x$. $\partial E$ is the boundary of a subset $E$ of $\mathbb{R}^n$.\\ For any continuous function $f : [t_0,t_1] \rightarrow \mathbb{R}^n$, let $\|f\|_{\infty} = \max_{t\in[t_0,t_1]} |f(t)|$. When $f$ is Lebesgue integrable, let $\|f\|_{\mathcal{L}^1([t_0,t_1])} = \int_{t_0}^{t_1} |f(t)|\ d t$. $W^{1,1}\left([t_0,T]; \mathbb{R}^n\right)$ is the set of all absolutely continuous functions $x:[t_0,T] \rightarrow \mathbb{R}^n$.\\
Consider now a real-valued function $f: \Omega \subset \mathbb{R}^n \rightarrow \mathbb{R}$, where $\Omega$ is an open set, and suppose that $f$ is locally Lipschitz. We denote by $\nabla f(\cdot)$ its gradient, which exists a.e. in $\Omega$. A vector $\zeta$ is in the \emph{reachable gradient} $\partial^* f(x)$ of $f$ at $x\in \Omega$ if there exists a sequence $\lbrace x_i \rbrace \subset \Omega$ such that $f$ is differentiable at $x_i$ for all $i \in \mathbb{N}$ and 
$$ x=\lim_{i \rightarrow \infty} x_i, \quad \zeta = \lim_{ i \rightarrow \infty} \nabla f(x_i).$$
Furthermore, the \emph{(Clarke's) generalized gradient} of $f$ at $x\in\Omega$, denoted by $\partial f(x)$, is the set of all the vectors $\zeta$ such that
\begin{equation}\label{ClarkeGrad}
\langle \zeta , v \rangle \leq \limsup_{
\begin{tiny}
\begin{array}{l} 
y \rightarrow x,\\ 
h \rightarrow 0^+
\end{array}
\end{tiny} 
} \frac{f(y+h v )- f(y) }{ h},\quad \forall v \in \mathbb{R}^n.
\end{equation}
It is known that $co( \partial^* f(x))=\partial f (x)$, where $co(A)$ denotes the convex hull of a subset $A$ of $\mathbb{R}^n$.\\ 
Let $f:\Omega \subset\mathbb{R}^n \rightarrow \mathbb{R}$ be any real-valued function defined on a open set $\Omega\subset\mathbb{R}^n$. For any $x\in\Omega$, the sets 
\[
\partial^- f(x)=\left\lbrace p\in\mathbb{R}^n : \liminf_{y \rightarrow x} \frac{f(y)-f(x)-\langle p, y - x\rangle  }{\mid y - x \mid}\geq 0 \right\rbrace,
\]
\[
\partial^+ f(x)=\left\lbrace p\in\mathbb{R}^n : \limsup_{y \rightarrow x} \frac{f(y)-f(x)-\langle p, y - x\rangle  }{\mid y - x \mid}\leq 0 \right\rbrace
\]
are called the \emph{(Fr\'echet) subdifferential} and \emph{superdifferential} of $f$ at $x$, respectively. A vector $p \in \mathbb{R}^n$ is a \emph{proximal supergradient} of $f$ at $x\in\Omega$ if there exist two constants $c,\rho \geq 0$ such that 
$$ f(y)-f(x)-\langle p, y-x \rangle \leq c \vert y-x \vert^2,\  \forall y \in B(x,\rho).$$ 
The set of all proximal supergradients of $f$ at $x$ is called the \emph{proximal superdifferential} of $f$ at $x$, and is denoted by $\partial^+_{pr} f(x)$. Note that $\partial^+_{pr} f(x)$ is a subset of the Fr\'echet superdifferential of $f$ at $x$.\\ %We can define the \emph{proximal superdifferential} of $f$ at $x$, denoted by $\partial^+_{pr} f(x)$, by reversing the inequality and the sign of the quadratic term in the above definition. 
For a mapping $f : \mathbb{R}^n \times \mathbb{R}^m \rightarrow \mathbb{R}$, associating to each $x \in \mathbb{R}^n$ and $y\in \mathbb{R}^m$ a real number, $\nabla_x f$, $\nabla_y f$ are its partial derivatives (when they do exist). The partial generalized gradient or partial Fr\'echet/proximal sub/superdifferential will be denoted in a similar way.\\ 
Let $\Omega$ be an open subset of $\mathbb{R}^n$. $C^1(\Omega)$ and $C^{1,1}(\Omega)$ are the spaces of all the functions with continuous and Lipschitz continuous first order derivatives on $\Omega$, respectively. \\ 
Let $K \subset \mathbb{R}^n$ be a convex set. For $\overline{v}\in K$, recall that the \emph{normal cone to $K$ at $\overline{v}$} (in the sense of convex analysis) is the set 
$$N_K(\overline{v})= \lbrace p\in \mathbb{R}^n: \langle p, v-\overline{v}\rangle\leq 0, \forall v\in K \rbrace.$$ 
A well-known separation theorem implies that the normal cone at any $v \in \partial K$ contains a half line. Moreover, if $K$ is not a singleton and has a $C^1$ boundary when $n>1$, then all normal cones at the boundary points of $K$ are half lines.\\ Finally, recall that a set-valued map $F:X \rightrightarrows Y$ is \emph{strongly injective} if $F(x)\cap F(y)=\emptyset$ for any two distinct points $x, y \in X$.
\subsection{Locally semiconcave functions}
Here, we recall the notion of semiconcave function in $\mathbb{R}^n$ and list some results useful in this paper. Further details may be found, for instance, in \cite{MR2041617}.\\

We write $[x, y]$ to denote the segment with endpoints $x, y$ for any $x, y \in \mathbb{R}^n$. 
\begin{definition}
Let $A\subset \mathbb{R}^n$ be an open set. We say that a function $u : A\rightarrow \mathbb{R} $ is (linearly) \emph{semiconcave} if it is continuous in $A$ and there exists a constant $c$ such that
\[ u(x + h) + u(x - h) - 2 u(x) \leq c | h |^2, \]
for all $x, h\in \mathbb{R}^n$ such that $[ x - h, x + h] \subset A$. The constant $c$ above is called a semiconcavity constant for $u$ in $A$. We denote by $SC(A)$ the set of functions which are semiconcave in $A$. We say that a function $u$ is semiconvex on $A$ if and only if $-u$ is semiconcave on $A$.
\end{definition}\label{PropSemiconc}
Finally recall that $u$ is locally semiconcave in $A$ if for each $x\in A$ there exists an open neighborhood of $x$ where $u$ is semiconcave.\\
In the literature, semiconcave functions are sometimes defined in a more general way. However, in the sequel we will mainly use the previous definition and properties recalled in the following proposition.
\begin{proposition}
Let $A\subset \mathbb{R}^n$ be an open set, let $u : A \rightarrow \mathbb{R}$ be a semiconcave function with a constant of semiconcavity $c$, and let $x \in A$. Then, 
\begin{enumerate}
\item a vector $p\in \mathbb{R}^n$ belongs to $\partial u(x)$ if and only if
\begin{equation}\label{Booo}
u(y) - u(x) - \langle p, y - x \rangle  \leq c |y - x|^2 
\end{equation}
for any point $y \in A$ such that $[y, x] \subset A$. Consequently, $\partial^{+} u(x)=\partial^{+,pr} u(x)$.
\item $\partial u(x) = \partial^{+} u(x)=co( \partial^{\ast} u(x))$.
\item If $\partial^+ u(x)$ is a singleton, then $u$ is differentiable at $x$.
\end{enumerate}
\end{proposition}
If $u$ is semiconvex, then (\ref{Booo}) holds reversing the inequality and the sign of the quadratic term and the other two statements are true with the Fr\'echet/proximal subdifferential instead of the Fre\'chet/proximal superdifferential. \\ 
In proving our main results we shall require the semiconvexity on the map $x \mapsto H(x,p)$. Let us recall a consequence which we will use later on.
\begin{lemma}[{\cite[Corollary~1]{MR2728465}}]\label{H2}
Suppose that $H$ is locally Lipschitz and the map $x\mapsto H(x,p)$ is locally semiconvex, where $H$ is as in (\ref{H}). Then, 
$$\partial H(x,p) \subset \partial_x^- H(x,p)\times \partial_p H(x,p), \quad \forall x\in\mathbb{R}^n, p  \in\mathbb{R}^n\setminus \lbrace 0 \rbrace.$$ 
\end{lemma}
%\begin{proof}
%For the proof of the first point we refer to Corollary 1 in \cite{MR2728465}. Regarding the second point, suppose now $\zeta\in \partial H(x,\lambda p)$ for some $\lambda>0$ and $p\neq 0$. By Proposition \ref{PropSemiconc} it holds if and only if 
%\begin{equation}
%H(y,\lambda p) - H(x,\lambda p) -\lambda \langle p, y - x \rangle  \geq  - c \lambda\mid p \mid |y - x|^2 
%\end{equation}
%for each $y$ in some neighborhood of $x$. Dividing for $\lambda >0$ and recalling the positive homogeneity of $H$ respect with the second variable, we see that it is equivalent to say $\frac{\zeta}{\lambda}\in \partial_x H(x,p)$, that is $\zeta \in\lambda \partial_x H(x,p)$.
%\end{proof}
\subsection{Differential Inclusions and Standing Assumptions}
We recall that the Hausdorff distance between two compact sets $A_i \subset \mathbb{R}^n,\ i=1,2$, is 
\[
dist_{\mathcal{H}}(A_1,A_2)=\max \lbrace dist^+_{\mathcal{H}}(A_1,A_2),dist^+_{\mathcal{H}}(A_2,A_1)  \rbrace,
\]
where $dist^+_{\mathcal{H}}(A_1,A_2)=\inf\lbrace \epsilon: A_1 \subset A_2 + B(0,\epsilon)  \rbrace$ is the semidistance. We say that a multifunction $F:\mathbb{R}^n \rightrightarrows \mathbb{R}^n$ with nonempty and compact values is locally Lipschitz if for each $x\in \mathbb{R}^n$ there exists a neighborhood $K$ of $x$ and a constant $c>0$ depending on $K$ so that $dist_{\mathcal{H}}(F(z),F(y))\leq c \mid z-y \mid $ for all $z, y \in K$.\\

Throughout this paper, we assume that the multifunction $F$ satisfies a collection of classical conditions of the theory of differential inclusions, the so-called \emph{Standing Hypotheses}:
$$\mbox{ \textbf{(SH)} }
\left\{\begin{array}{ll} 
(i) & F(x) \mbox{ is nonempty, convex, compact for each } x \in \mathbb{R}^n,\\
(ii) & F \mbox{ is locally Lipschitz with respect to the Hausdorff metric},\\
(iii)& \exists r>0 \mbox{ so that } \max \lbrace \vert v \vert : v\in F(x)\rbrace \leq r (1+\vert x \vert)\, \forall x \in \mathbb{R}^n.
\end{array}\right.$$

Assumptions (SH)(i)-(ii) guarantee the existence of local solutions of (\ref{May1})-(\ref{May2}) and  (SH)(iii) guarantees that solutions are defined on $[t_0,T]$.\\

For the sake of brevity, we usually refer to the Mayer problem (\ref{May1})-(\ref{May2})-(\ref{Mayer}) as $\mathcal{P}(t_0,x_0)$. Assuming (SH) and $\phi$ lower semicontinuous implies that $\mathcal{P}(t_0,x_0)$ has at least one \emph{optimal solution}, that means a solution $\overline{x}(\cdot)\in W^{1,1}\left([t_0,T]; \mathbb{R}^n\right)$ of (\ref{May1}) satisfying (\ref{May2}) such that 
\[
\phi(\overline{x}(T))\leq \phi (x(T)),
\]
for any trajectory $x(\cdot)\in W^{1,1}\left([t_0,T]; \mathbb{R}^n\right)$ of (\ref{May1}) satisfying (\ref{May2}). Actually, the Standing Hypotheses were first introduced with the only property of upper semicontinuity of $F$ instead of (SH)(ii) and that would be enough to deduce the existence of optimal trajectories, but in this paper we will often take advantage of the local Lipschitzianity of $F$. For the basics of the theory of differential inclusions we refer, e.g., to \cite{MR755330}.\\

Under assumption (SH) one can show that it is possible to associate with each optimal trajectory $x(\cdot)$ for $\mathcal{P}(t_0,x_0)$ an arc $p(\cdot)$ such that the pair $(x(\cdot),p(\cdot))$ satisfies a Hamiltonian inclusion. See, e.g., \cite[Theorem~3.2.6]{MR709590}.
%Recall that the Hamiltonian associated to the problem \eqref{May1}-\eqref{Mayer} is the function $H:\mathbb{R}^n\times \mathbb{R}^n \rightarrow \mathbb{R}$ defined as
%\begin{equation}\label{H}
%H(x,p)= \sup_{v \in F(x)} \langle v,p\rangle .
%\end{equation}
\begin{theorem}\label{TheoDualArc}
Assume (SH) and that $\phi:\mathbb{R}^n \rightarrow \mathbb{R}$ is locally Lipschitz. If $x(\cdot)$ is an optimal solution for $\mathcal{P}(t_0,x_0)$, then there exists an arc $p:[t_0,T]\rightarrow \mathbb{R}^n$ which, togheter with $x(\cdot)$, satisfies
\begin{equation}\label{AdjointSystem}
(-\dot{p}(s),\dot{x}(s)) \in \partial H(x(s),p(s)), \mbox{ a.e. }\ s\in \left[ t_0, T \right], 
\end{equation}
and 
\begin{equation}\label{TC}
-p(T)\in \partial \phi(x(T)). 
\end{equation}
\end{theorem}
Given an optimal trajectory $x(\cdot)$, any arc $p(\cdot)$ satisfying the \emph{adjoint system} (\ref{AdjointSystem}) and the \emph{tranversality condition} (\ref{TC}) is called a \emph{dual arc} associated with $x(\cdot)$. Furthermore, if $(q,v)$ belongs to $\partial H(x,p)$, then $v\in F(x)$ and $\langle p, v \rangle = H(x,p)$. Thus the system \eqref{AdjointSystem} encodes the equality
\begin{equation}\label{MAXIMUMprincipleEQUATION}
H(x(t),p(t))= \langle p(t), \dot{x}(t) \rangle \mbox{ for a.e. } t \in [t_0,T].
\end{equation}
This equality shows that the scalar product $\langle v,p(t) \rangle$ is maximized over $F(x(t))$ by $v=\dot{x}(t)$. For this reason, the previous result is known as the \emph{maximum principle} (in Hamiltonian form).
\begin{remark}\label{RemarkDualArc}
If the dual arc introduced in Theorem \ref{TheoDualArc} is equal to zero at some time $t \in [t_0,T]$, then it is equal to zero at every time. Indeed, consider a compact set $K \subset \mathbb{R}^n$ containing $\overline{x}([t_0,T])$. If we denote by $c_K$ the Lipschitz constant of $F$ on $K$, it follows that $c_K \vert p\vert$ is the Lipschitz constant for $H(\cdot, p)$ on the same set. Indeed, let $x,y\in K$ and $v_x$ be such that $H(x,p)= \langle v_x,p \rangle$. By (SH), there exists $v_y \in F(y)$ such that
\begin{equation}
H(x,p)-H(y,p)\leq \langle v_x - v_y , p \rangle \leq c_K \vert p \vert \vert x - y \vert.
\end{equation}
Recalling (\ref{ClarkeGrad}), it follows that
\begin{equation}\label{Impoooo}
\vert \zeta \vert \leq c_K \vert p \vert, \ \forall \zeta \in \partial_x H(x,p), \forall x \in K, \forall p \in \mathbb{R}^n. 
\end{equation} 
Hence, in view of the differential inclusion verified by $\overline{p}(\cdot)$, 
\begin{equation}\label{Gronwall}
\vert \dot{\overline{p}}(s)\vert\leq c_K \vert \overline{p}(s) \vert, \mbox{  for a.e. } s \in [t_0,T]. 
\end{equation}
By Gronwall's Lemma, we obtain that either $\overline{p}(s)\neq 0$ for every $s\in[t_0,T]$, or $\overline{p}(s)= 0$ for every $s\in[t_0,T]$.\\
\end{remark}
Recall now that the \emph{value function} $V:[0,T] \times \mathbb{R}^n \rightarrow \mathbb{R}$ associated to the Mayer problem is defined by: for all $(t_0,x_0)\in [0,T]\times\mathbb{R}^n$ 
\begin{equation}\label{ValueFunction}
V(t_0,x_0)= \inf \left\lbrace \phi(x(T)) : x \in W^{1,1}\left([t_0,T]; \mathbb{R}^n\right) \mbox{ satisfies } (\ref{May1}) \mbox{ and }(\ref{May2}) \right\rbrace.
\end{equation} 
As far as $V$ is concerned, recall that, under assumptions (SH), $V$ is locally Lipschitz and solves in the viscosity sense the Hamilton-Jacobi equation
\begin{equation}\label{HJBeq}
\left\{\begin{array}{l} 
-\partial_t u(t,x) + H(x,-u_x (t,x) )=0\quad \mbox{ in } (0,T)\times \mathbb{R}^n,\\
u(T,x)=\phi(x), \mbox{ } x\in \mathbb{R}^n,
\end{array}\right.
\end{equation}
where $H$ is the Hamiltonian associated to $F$. Indeed, if the multifunction $F$ satisfies assumption (SH), then it always admits a parameterization by locally Lipschitz function (see, e.g., \cite[Theorem~7.9.2]{MR1048347}) and the result is well-known for the Lipschitz-parametric case (see, e.g., \cite{MR2041617}).
\begin{proposition}\label{LemmaSolu}
Assume (SH) and that $\phi:\mathbb{R}^n \rightarrow \mathbb{R}$ is locally Lipschitz. Then the value function of the Mayer problem is the unique viscosity solution of the problem \eqref{HJBeq}, where the Hamiltonian $H$ is given by \eqref{H}.
\end{proposition}
We conclude this part recalling that $V$ satisfies the \emph{dynamic programming principle}. Hence, if $y(\cdot)$ is any trajectory of the system \eqref{May1}-\eqref{May2}, then the function $s\rightarrow V(s,y(s))$ is nondecreasing, and it is constant if and only if $y(\cdot)$ is optimal.
\subsection{Sufficient conditions for optimality}
In the control literature, it is well known that the full sensitivity relation involving the Fr\'echet superdifferential of $V$, coupled with the maximum principle, is a sufficient condition for optimality. For reader's convenience, we shall recall this result in our context. The proof uses the same arguments of Theorem $4.1.$ in \cite{Cannarsa:1991:COT:120771.118946} and it is omitted here. In  \cite{Cannarsa:1991:COT:120771.118946}, the authors have used the fact that (see, e.g., \cite{MR1048347})
\begin{equation}\label{DiniDer0}
\begin{split}
\partial^+ V(t, \overline{x}(t))&=\left\lbrace (p',p'') \in\mathbb{R}\times\mathbb{R}^n :\right. \\
&\left. \forall (\theta',\theta'')\in\mathbb{R}\times\mathbb{R}^n, D_{\downarrow} V (t,\overline{x}(t))(\theta',\theta'') \leq p' \theta'+ \langle p'',\theta'' \rangle \right\rbrace,
\end{split}
\end{equation}
where the \emph{upper Dini derivative} of $V$ at $(t,\overline{x}(t))$ in the direction $(\theta',\theta'')$ is given by 
\begin{equation}\label{DiniDer}
D_{\downarrow} V (t,\overline{x}(t))(\theta',\theta''):= \limsup_{\tau \rightarrow 0^+} \frac{V(t+ \tau \theta' , \overline{x}(t)+\tau \theta'') - V(t, \overline{x}(t))}{\tau }.
\end{equation}
\begin{theorem}\label{Sufficient}
Assume (SH) and let $x:[t_0,T]\rightarrow \mathbb{R}^n$ be a solution of the system \eqref{May1}-\eqref{May2}. If, for almost every $t\in [t_0,T]$, there exists $p(t)\in\mathbb{R}^n$ such that
\begin{equation}\label{hhh}
\begin{array}{c}
\langle p(t),\dot{x}(t)\rangle =H(x(t),p(t)),\\
( H(x(t),p(t)),-p(t)) \in \partial^+ V(t,x(t)),
\end{array}
\end{equation}
then $x$ is optimal for problem $\mathcal{P}(t_0,x_0)$. 
\end{theorem}
%\begin{proof}
%Consider the absolutely continuous function $\psi(t):=V(t,x(t))$. Let $t\in [t_0,T]$ be such that $\dot{\psi}(t)$ and $\dot{x}(t)$ do exist and \eqref{hhh} holds true. From \eqref{DiniDer0}, \eqref{DiniDer} and \eqref{hhh}, we deduce that
%\[
%0=\langle ( H(x(t),p(t)),-p(t)) ,(1,\dot{x}(t))\rangle \geq D_{\downarrow} V(t,x(t))(1,\dot{x}(t))
%\]
%\[
%\limsup_{h \rightarrow 0^+} \frac{V(t+h,x(t)+h \dot{x}(t))-V(t,x(t))}{h}
%\]
%\[
%\limsup_{h \rightarrow 0^+} \frac{V(t+h,x(t+h))-V(t,x(t))}{h}=\dot{\psi}(t).
%\]
%This implies that $\psi$ is nonincreasing. Thus the dynamic programming principle guarantees that $x$ is optimal for $\mathcal{P}(t_0,x_0)$.
%\end{proof}
\subsection{Main assumptions}
We impose further conditions on the Hamiltonian associated to $F$. For each nonempty, convex and compact subset $K\subseteq \mathbb{R}^n$,
$$\mbox{ \textbf{(H1)} }
\left\{ 
\begin{array}{ll}
(i) & \exists\ c \geq 0 \mbox{ so that } x \mapsto H(x,p)\mbox{ is semiconvex on K with constant}
\\  & c \mid p\mid, \\
(ii) & \mbox{the gradient }\nabla_p H(x,p) \mbox{ exists and is locally Lipschitz in } x \mbox{ on}\\
 & K, \mbox{ uniformly over }p \mbox{ in any compact subset of }\mathbb{R}^n \smallsetminus \lbrace 0\rbrace.
\end{array}\right.$$\\ %Indeed, in \cite{MR2728465} the authors have also given a method to generate many examples of multifunctions satisfying (SH) and (H1), without admitting a smooth parameterization.\\
Some examples of multifunctions satisfying
(SH) and (H1) are given in \cite{MR2728465}.\\

Let us start by analyzing the meaning of the assumption (H1)(i). The semiconvexity of the map $x \mapsto H(x,p)$ on a convex compact subset $K$ of $\mathbb{R}^n$ with constant $c \mid p \mid$ is equivalent to the \emph{mid-point property} of the multifunction $F$ on $K$, that is
\[
dist^+_{\mathcal{H}} \left( 2 F(x), F(x+z)+F(x-z) \right) \leq c \mid z \mid^2,
\]
for all $x, z$ so that $x, x\pm z\in K$. A consequence of the above hypotheses is that the generalized gradient of $H$ splits into two components, as described in Lemma \ref{H2}. This implies that the adjoint system \eqref{AdjointSystem} takes the form \eqref{HI2} below.
\begin{theorem}[{\cite[Corollary~2]{MR2728465}}]\label{TheoDualArc2}
Assume that (SH) and (H1)(i) hold and $\phi:\mathbb{R}^n \rightarrow \mathbb{R}$ is locally Lipschitz. If $x(\cdot)$ is an optimal solution for $\mathcal{P}(t_0,x_0)$, then there exists an arc $p:[t_0,T]\rightarrow \mathbb{R}^n$ which, together with $x(\cdot)$, satisfies
\begin{equation}\label{HI2}
\left\{\begin{array}{rll}
-\dot{p}(s) &\in& \partial_x^- H(x(s),p(s)), \\
\dot{x}(s) &\in& \partial_p H(x(s),p(s)), 
\end{array}\right.
\mbox{ a.e. }\ s\in \left[ t_0, T \right]
\end{equation}
and
\begin{equation}\label{TC2}
-p(T)\in \partial \phi(x(T)). 
\end{equation}
\end{theorem}
Concerning the assumption (H1)(ii), the existence of the gradient of $H$ with respect to $p$ is equivalent to the fact that the argmax set of $\langle v,p \rangle$ over $v\in F(x)$ is the singleton $\lbrace \nabla_p H(x,p)\rbrace$, for each $p\neq 0$. Thus, the following relation holds:
\begin{equation}\label{derive}
H(x,p)= \langle \nabla_p H(x,p), p \rangle,\ \forall p \neq 0.
\end{equation}
Moreover, it is easy to see that, for every $x$, the boundary of the sets $F(x)$ contains no line segment.\\
The main impact of the local Lipschitzianity of the map $x \mapsto\nabla_p H(x,p)$ is the following result, whose proof is straightforward.
\begin{lemma}[{\cite[Proposition~3]{MR2728465}}]\label{LemmmaLip}
Assume (SH) and (H1). Let $p:[t,T]\rightarrow \mathbb{R}^n$ be an absolutely continuous arc with $p(s)\neq 0$ for all $s\in [t,T]$. Then, for each $x \in \mathbb{R}^n$, the Cauchy problem
\begin{equation}\label{Cau}
\left\{\begin{array}{l}
\dot{y}(s)= \nabla_p H ( y(s),p(s)) \quad \mbox{ for all } s\in \left[ t, T \right],\\
y(t)= x,
\end{array}\right.
\end{equation}
has a unique solution $y(\cdot;t,x)$. Moreover, there exists a constant $k$ such that
\begin{equation}
| y(s;t,x)-y(s;t,z)|\leq e^{k (T-t)} |z-x|,\ \forall z,x\in\mathbb{R}^n, \forall s \in [t,T].
\end{equation}
\end{lemma}
\begin{remark} \label{RemarkContinuity}
Note that the map $p \mapsto \nabla_p H(x,p)$ is continuous for $p\neq 0$. Thus, the Lipschitzianity of the map $x \mapsto \nabla_p H(x,p)$ implies that $(x,p)\mapsto\nabla_p H(x,p)$ is a continuous map, for $p\neq 0$. This is the reason why the system of ODEs \eqref{Cau} is verified everywhere on $[t,T]$,  not just almost everywhere.
Suppose now  $x(\cdot)$ is optimal for $\mathcal{P}(t_0,x_0)$ and $p(\cdot)$ is any nonvanishing dual arc associated with $x(\cdot)$---if they do exist. Then, Lemma \ref{LemmmaLip} implies that $x(\cdot)$ is the unique solution of the Cauchy problem (\ref{Cau}) with $t=t_0$, $x(t_0)= x_0$, and $p(\cdot)$ equal to such a dual arc. Furthermore, in this case, $x(\cdot)$ is of class $C^1$ and the maximum principle \eqref{MAXIMUMprincipleEQUATION} holds true for all $t\in [t_0,T]$.
\end{remark}
\subsection{$R$-convex sets}
Let $A$ be a compact and convex subset of $\mathbb{R}^n$ and $R>0$.
\begin{definition}
The set $A$ is \emph{$R-$convex} if, for each $z, y \in \partial A$ and any vectors $n\in N_A(z), m\in N_A(y)$ with $\mid n\mid =\mid m \mid=1$, the following inequality holds true
\begin{equation}
\mid z-y \mid \leq R \mid n-m \mid. 
\end{equation}
\end{definition}
The concept of $R-$convex set is not new. It is a special case of \emph{hyperconvex sets} (with respect to the ball of radius R and center zero) introduced by Mayer in \cite{MR1545515}. A study of hyperconvexity appears also in \cite{MR1505058,Pasq}. The notion of $R$-convexity was considered, among others, by Levitin, Poljak, Frankowska, Olech, Pli\'s, Lojasiewicz, and Vian (they called these sets
$R$-regular, $R$-convex, as well strongly convex). We first recall some interesting characterizations of $R-$convex sets.
\begin{proposition}[{\cite[Proposition 3.1]{MR648458}}]\label{equiv}
Let $A$ be a compact and convex subset of $\mathbb{R}^n$. Then the following conditions are equivalent
\begin{enumerate}
\item $A$ is $R-$convex,
\item $A$ is the intersection of a family of closed balls of radius $R$,
\item for any two points $x , y \in \partial A$ such that $|x-y|\leq 2 R$, each arc of a circle of radius $R$ which joins $x$ and $y$ and whose lenght is not greater that $\pi R$ is contained in $A$,
\item for each $z \in \partial A$ and any $n \in N_{A}(z), \mid n \mid=1$, the ball of center $z-R n$ and radius $R$ contains $A$, that is $\mid z - R n - x \mid \leq R$ for each $x \in A$,
\item for each $z\in \partial A$ and any vector $n\in N_{A}(z)$ with $|n|=1$, we have the inequality
\begin{equation}
|z-x|\leq \sqrt{2 R} \langle z-x,n \rangle^{\frac{1}{2}},\ \forall x\in A.
\end{equation} 
\end{enumerate}
\end{proposition}
$R$-convex sets are obviously convex. Moreover, the boundary of an $R$-convex set $A$  satisfies a generalized lower bound for the curvature, even though $\partial A$ may be a nonsmooth set. Indeed, for every point $x \in \partial A$ there exists a closed ball $B_x$ of radius $R$ such that $x\in \partial B_x$ and $A \subset B_x$. This fact suggests that, in some sense, the curvature of $\partial A$ is bounded below by $1/R$.

\begin{definition}
A multifunction $F:\mathbb{R}^n \rightrightarrows \mathbb{R}^n $ is \emph{locally strongly convex} if for each compact set $K\subset\mathbb{R}^n$ there exists $R>0$ such that $F(x)$ is $R$-convex for every $x\in K$.
\end{definition}
We can reformulate the above property of $F$ in an equivalent Hamiltonian form. Here, we denote by $F_p(x)$ the argmax set of $\langle v,p\rangle$ over $v\in F(x)$. The existence of $\nabla_p H(x,p)$ is equivalent to the fact that the set $F_p(x)$ is the singleton $\lbrace \nabla_p H(x,p)\rbrace$. 
\[
\mbox{\textbf{(H2)}}\left\{
\begin{array}{l}
\mbox{For every compact } K \subset \mathbb{R}^n, \mbox{ there exists a constant }	\\ c'=c'(K)>0 \mbox{ such that for all } x \in K,\ p \in \mathbb{R}^n,\mbox{ we have}: \\
v_p\in F_p(x) \Rightarrow  \langle v- v_p ,p \rangle \leq - c' | p | |v-v_p|^{2},\ \forall v \in F(x).
\end{array}
\right.
\] 
In the next lemma we show that the local strong convexity of $F$ is equivalent to assumption (H2) for the associated Hamiltonian, giving also a result connecting (H2) with the regularity of $H$. 
\begin{lemma} 
Suppose $F:\mathbb{R}^n \rightrightarrows \mathbb{R}^n$ is a multifunction satisfying
(SH). Let $K$ be any convex and compact subset of $\mathbb{R}^n$. Then
\begin{enumerate}
\item (H2) holds with a constant $c'$ on $K$ if and only if $F(x)$ satisfies the $R$-convexity property for all $x\in K$ with radius $R=(2 c')^{- 1}$. 
\item  If (H2) holds, then $\nabla_p H(x,p)$ exists for all $x \in K$ and $p\in \mathbb{R}^n\smallsetminus\lbrace 0 \rbrace$ and is H\"older continuous in $x$ on $K$ with exponent $1/2$, uniformly for $p$ in any compact subset of $\mathbb{R}^n \smallsetminus \lbrace 0 \rbrace$.
\end{enumerate}
\end{lemma}
\begin{proof}
For all $x\in K$ and $v \in \partial F(x)$, we have $v\in F_{y_v}(x)$ for all $y_v \in N_{F(x)}(v)$. Therefore, (H2) holds with constant $c'$ on $K$ if and only if for any $y_v\in N_{F(x)}(v)$ with $\mid y_v \mid=1$ we have 
\[ \langle  v - \overline{v} ,y_v \rangle \geq  c' |v-\overline{v} |^{2},\ \forall \overline{v} \in F(x), \]
or equivalently, 
\[
\mid v- \overline{v} \mid\leq  \sqrt{ 2 \frac{1}{2 c'}} \langle v-\overline{v}, y_v \rangle^{\frac{1}{2}},\ \forall \overline{v} \in F(x),
\]
for all $y_v \in N_{F(x)}(v)$ with $\mid y_v \mid=1$. By Proposition \ref{equiv}, this is equivalent to the $(2c')^{-1}-$convexity of $F(x)$ for each $x \in K$. For the proof of the second statement we refer to \cite[Proposition~4]{MR2728465}.
\end{proof}
The second statement of the above lemma is not an equivalence, in general, as is shown by the example below.  Moreover, assumption (H1) does not follow from (H2).

\begin{example}
Let us denote by $M \subset \mathbb{R}^2$ the intersection of the epigraph of the function $f:\mathbb{R}\rightarrow\mathbb{R},\ f(x)=x^4$, and the closed ball $B(0,R),~R>0$. Let us consider the multifunction $F:\mathbb{R}^2 \rightrightarrows\mathbb{R}^2$ that associates set $M$ with any $x\in \mathbb{R}^2$. Observe that $M$ fails to be strongly convex, since the curvature at $x=0$ is equal to zero. Moreover, since $M$ is a closed convex set and its boundary contains no line, the argmax of $\langle v,p \rangle$ over $v\in M$ is a singleton for each $p\neq 0$. So, the Hamiltonian $H(p)=\sup_{v\in M} \langle v,p\rangle$ is differentiable for each $p\neq 0$. Note that the gradient $\nabla_p H$ is constant with respect to the $x$ variable. Consequently, the Hamiltonian satisfies (H1) and not (H2).
\end{example}

\section{Sensitivity relations}

%In the literature, sensitivity relations are usually understood as properties with connect the slope of the value function to optimal trajectories. Technically, such properties take the form of inclusions of dual arc into suitable generalized gradients of $V$. For this reason, sensitivity relations are often referred to as supergradient inclusions. In some sense, these inclusions can be viewed as the backward propagation of the superdifferentiability of the value function along an optimal trajectories.\\
%For the Mayer problem, supergradient inclusions were obtained for parameterized systems in the case of the Fr\'echet superdifferential (see e.g. \cite{Cannarsa:1991:COT:120771.118946}, \cite{MR2041617}). In this section, we will show that a similar result holds for differential inclusions in the case of both proximal and Fr\'echet superdifferentials. (now this part is in the introduction)
In this section we discuss sensitivity relations. First we prove, under suitable assumptions, the validity of the partial sensitivity relations and then the full sensitivity relations, involving both Fr\'echet and proximal superdifferential of the value function. %As we shall see, the proofs, which are completely different from the case of optimal control problems, make an essential use of the maximum principle as in the approach of \cite{MR2728465}.  
\begin{theorem}\label{Lemma_super_nonprox}
Assume (SH), (H1) and let $\phi: \mathbb{R}^n \rightarrow \mathbb{R}$ be locally Lipschitz. Let $\overline{x}:[t_0,T]\rightarrow \mathbb{R}^n$ be an optimal solution for $\mathcal{P}(t_0,x_0)$ and $\overline{p}:[t_0,T]\rightarrow \mathbb{R}^n$ be any solution of the differential inclusion
\begin{equation}
\left\{\begin{array}{rll}
-\dot{p}(t)&\in&\partial_x^- H(\overline{x}(t),p(t)) \\
\dot{\overline{x}}(t)&\in&\partial_p H(\overline{x}(t),p(t)) 
\end{array}\right. 
\quad \mbox{ a.e. in } \left[ t_0, T \right],
\end{equation}
with the transversality condition
\begin{equation}\label{dualarc2}
-\overline{p}(T)\in \partial^{+,pr} \phi(\overline{x}(T)).
\end{equation}
Then, there exist constants $c_0,r>0$ such that, for all $t\in[t_0,T]$ and all $h \in B(0,r)$, 
\begin{equation}\label{iclu_super}
V(t,\overline{x}(t)+h)-V(t,\overline{x}(t)) \leq\ \langle -\overline{p}(t),h\rangle + c_0  \mid h \mid^2. 
\end{equation} 
Consequently, $\overline{p}(\cdot)$ satisfies the proximal partial sensitivity relation 
\begin{equation}\label{PPSR}
-\overline{p}(t)\in \partial_x^{+,pr}V(t,\overline{x}(t))\mbox{ for all }t\in[t_0,T].
\end{equation}
\end{theorem}

To prove the above theorem, we need the following lemma.
\begin{lemma}\label{LemmaTecnico}
Assume $\overline{p}(T)\neq 0$ and fix $t\in[t_0,T)$. For each $h\in B(0,1)$, let $x_{h}:[t,T]\rightarrow \mathbb{R}^n$ be the solution of the problem
\begin{equation}\label{ennesimoBo}
\left\{\begin{array}{l}
\dot{x}(s)= \nabla_p H(x(s),\overline{p}(s)), \quad  s\in \left[ t, T \right],\\
x(t)=\overline{x}(t)+h.
\end{array}\right.
\end{equation}
Then, there exist constants $c_1, K_1$, independent of $t\in[t_0,T)$, such that
\begin{equation}\label{Booooo3}
\parallel x_h - \overline{x} \parallel_{\infty} \leq e^{K_1 T }\mid h \mid,
\end{equation} 
and 
\begin{equation}\label{cambia2}
\langle \overline{p}(t),h\rangle  + \langle -\overline{p}(T),x_{h}(T)-\overline{x}(T)\rangle \leq c_1 \mid h \mid^2.
\end{equation}
\end{lemma}
\begin{proof}
Thanks to Remark \ref{RemarkDualArc}, we have that $\overline{p}(t)\neq 0$ for all $t\in [t_0,T]$. Hence, $\overline{x}(\cdot)$ is the unique solution of the Cauchy problem
\begin{equation}\label{AAAdu}
\left\{\begin{array}{l}
\dot{x}(s)= \nabla_p H(x(s),\overline{p}(s))\ \mbox{ for all } s\in \left[ t, T \right],\\
x(t)=\overline{x}(t).
\end{array}\right.
\end{equation}
Since $F$ has sublinear growth and $\nabla_p H(\cdot,p)$ is locally Lipschitz, by standard arguments based on Gronwall's Lemma we conclude that there exists a constant $K_1$, independent of $t\in[t_0,T)$, such \eqref{Booooo3} holds true.
In order to prove \eqref{cambia2}, note that
\[
\langle \overline{p}(t),h\rangle  + \langle -\overline{p}(T),x_{h}(T)-\overline{x}(T)\rangle = \int_{t}^T \frac{d}{d s} \langle -\overline{p}(s), x_{h}(s)-\overline{x}(s)\rangle ds  
\]
\[ 
= \int_{t}^T  \langle -\dot{\overline{p}}(s), x_{h}(s)-\overline{x}(s)\rangle ds\ + \int_{t}^T  \langle -\overline{p}(s), \dot{x}_{h}(s)-\dot{\overline{x}}(s)\rangle \ d s := (I) + (II). 
\]
Since $-\dot{\overline{p}}(s)\in \partial_x^- H(\overline{x}(s),\overline{p}(s))$ a.e. in $[t_0,T]$, we obtain
\[ (I)\leq \int_{t}^T  \left( c \mid \overline{p}(s)\mid\cdot \mid x_{h}(s)-\overline{x}(s)\mid^2+ H(x_{h}(s),\overline{p}(s))-H(\overline{x}(s),\overline{p}(s)) \right) \ d s 
\]
\[ \leq  \int_{t}^T  \left( c \mid\overline{p}(s)\mid\cdot\mid h \mid^2 e^{2 K_1 T}+ H(x_{h}(s),\overline{p}(s))-H(\overline{x}(s),\overline{p}(s))\right) \ d s,
\]
for some nonnegative constant $c$. Now recalling (\ref{derive}), (\ref{ennesimoBo}) and \eqref{AAAdu}, we get
\[ (II)=\int_{t}^T  \langle -\overline{p}(s), \nabla_p H(x_{h}(s),\overline{p}(s)) -\nabla_p H( \overline{x}(s),\overline{p}(s))\rangle ds  \]
\[
= \int_{t}^T   \left( - H(x_{h}(s),\overline{p}(s))+ H(\overline{x}(s),\overline{p}(s))\right) ds.
\]
Adding up the previous relations, it follows that there exists a constant $c_1$, independent of $t\in[t_0,T)$, such that \eqref{cambia2} holds true.
%\begin{equation}
%\langle \overline{p}(t),h\rangle  + \langle -\overline{p}(T),x_{h}(T)-\overline{x}(T)\rangle \leq c_1 \mid h \mid^2.
%\end{equation}
The lemma is proved.
\end{proof} 
\begin{proof}[Proof of Theorem \ref{Lemma_super_nonprox}]
First note that the estimate \eqref{iclu_super} is immediate for $t=T$. Let us observe that in view of Remark \ref{RemarkDualArc} the above dual arc $\overline{p}(\cdot)$ satisfies the following:
\begin{enumerate}
\item[(i)] either $\overline{p}(t)\neq 0$ for all $t\in [t_0,T]$,
\item[(ii)] or $\overline{p}(t)=0$ for all $t\in [t_0,T]$.
\end{enumerate}
We shall analyze each of the above cases separately. Suppose, first, that $\overline{p}(t)\neq 0$ for all $t\in [t_0,T]$ and fix $t \in [t_0,T)$. Then $\overline{x}(\cdot)$ is the unique solution of the Cauchy problem
\begin{equation}\label{AAA}
\left\{\begin{array}{l}
\dot{x}(s)= \nabla_p H(x(s),\overline{p}(s)) \ \mbox{ for all } s\in \left[ t, T \right],\\
x(t)=\overline{x}(t).
\end{array}\right.
\end{equation}
For each $h\in B(0,1)$, let $x_{h}(\cdot)$ be the solution of the problem \eqref{ennesimoBo}.
By the optimality of $\overline{x}(\cdot)$, the very definition of the value function, and the dynamic programming principle, we have that, for all $h\in B(0,1)$,
\begin{equation}\label{prima}
V(t,\overline{x}(t)+h)-V(t,\overline{x}(t)) +\langle \overline{p}(t),h \rangle \leq \phi(x_{h}(T))-\phi(\overline{x}(T))+\langle\overline{p}(t),h \rangle. 
\end{equation}
Moreover, owing to (\ref{dualarc2}), there exist constants $c,R>0$ so that, for all $z \in B(0,R)$,
\begin{equation}\label{Nuovo1}
\phi(\overline{x}(T)+z)-\phi(\overline{x}(T))\leq \langle -\overline{p}(T),z \rangle+c |z|^2.
\end{equation}
Observe that, by (\ref{Booooo3}), 
\begin{equation}\label{BBB}
\parallel x_h - \overline{x} \parallel_{\infty}<R,\ \forall\ h\in\mathbb{R}^n \text{ such that } \mid h \mid<r_1:=\min \lbrace 1, R e^{-K_1 T}\rbrace.
\end{equation}
Then, on account of (\ref{prima}), (\ref{Nuovo1}) and (\ref{BBB}), we conclude that for each $h\in B(0,r_1)$,
\begin{equation}\label{cambia}
\begin{split}
V(t,&\overline{x}(t)+h)-V(t,\overline{x}(t)) +\langle \overline{p}(t),h\rangle  
\\
&\leq \langle \overline{p}(t),h\rangle  + \langle -\overline{p}(T),x_{h}(T)-\overline{x}(T)\rangle  + c \mid x_{h}(T)-\overline{x}(T)\mid^2.
\end{split}
\end{equation}
Therefore, in view of (\ref{cambia}) and Lemma \ref{LemmaTecnico}, there exists a constant $c_2$, independent of $t$, so that for all $h \in B(0,r_1)$ and $t\in [t_0,T)$,
\begin{equation}\label{iclu_super2}
V(t,\overline{x}(t)+h)-V(t,\overline{x}(t))+ \langle \overline{p}(t),h\rangle \leq  c_2  \mid h \mid^2. 
\end{equation} 
The proof of (\ref{iclu_super}) in the case (i) is complete. \\
Next, suppose we are in case (ii), that is $p(t)=0$ for all $t\in [t_0,T]$. Let $t\in[t_0,T)$ be fixed. Then, by Filippov's Theorem (see, e.g., Theorem 10.4.1 in \cite{MR1048347}), there exist constants $r_2, K_2$, independent of $t\in[t_0,T]$, such that, for any $h \in\mathbb{R}^n$ with $\mid h \mid \leq r_2$, the initial value problem
\begin{equation}\label{lastAga2}
\left\lbrace
\begin{array}{l}
\dot{x}(s)\in F(x(s))\ \textit{ a.e. in } [t,T],\\
x(t)=\overline{x}(t)+h. 
\end{array}
\right.
\end{equation}
has a solution, $x_h(\cdot)$, that satisfies the inequality
\begin{equation}\label{Booooo3star}
\parallel x_h - \overline{x} \parallel_{\infty} \leq e^{K_2 T }\mid h \mid.
\end{equation} 
By the optimality of $\overline{x}(\cdot)$, the very definition of the value function, and the dynamic programming principle it follows that
\begin{equation}\label{Nuovo7}
V(t,\overline{x}(t)+h)-V(t,\overline{x}(t)) \rangle \leq \phi(x_{h}(T))-\phi(\overline{x}(T)). 
\end{equation}
Moreover, owing to (\ref{dualarc2}) an recalling that $\overline{p}(\cdot)$ is equal to zero at each point, there exist constants $c,R>0$ so that, for all $z \in B(0,R)$,
\begin{equation}\label{Nuovo8}
\phi(\overline{x}(T)+z)-\phi(\overline{x}(T))\leq c |z|^2.
\end{equation}
In view of \eqref{Booooo3star}, \eqref{Nuovo7} and \eqref{Nuovo8}, we obtain that there exists a constants $c_4$ such that, for all $t\in [t_0,T]$ and $h\in B(0,r_3)$ whit $r_3:=\min \lbrace r_2, R e^{-K_2 T} \rbrace$, it holds that
\begin{equation}
V(t,\overline{x}(t)+h)-V(t,\overline{x}(t)) \leq\ c_4  \mid h \mid^2. 
\end{equation} 
The proof is complete also in case (ii).
\end{proof}
\begin{remark}\label{RelarkPartialSensRelation}
One can easily adapt the previous proof to show that the above inclusion holds true with the Fr\'echet superdifferential as well, that is if $-\overline{p}(T)\in \partial^{+} \phi(\overline{x}(T))$, then $$-\overline{p}(t)\in \partial_{x}^{+}V(t,\overline{x}(t)) \mbox{ for all } t \in [t_0,T].$$ In this case, the term  $c_0 \mid h \mid^2$ in (\ref{iclu_super}) is replaced by $o (\mid x_h(T)-\overline{x}(T)\mid )$.
\end{remark}
\begin{theorem}\label{FullSensivTheorem}
Assume (SH), (H1) and let $\phi: \mathbb{R}^n \rightarrow \mathbb{R}$ be locally Lipschitz. Let $\overline{x}:[t_0,T]\rightarrow \mathbb{R}^n$ be an optimal solution for the problem $\mathcal{P}(t_0,x_0)$ and $\overline{p}:[t_0,T]\rightarrow \mathbb{R}^n$ be any solution of the differential inclusion
\begin{equation}\label{DualArc}
\left\{\begin{array}{rll}
-\dot{p}(t)&\in&\partial_x^- H(\overline{x}(t),p(t)) \\
\dot{\overline{x}}(t)&\in&\partial_p H(\overline{x}(t),p(t)) 
\end{array}\right. 
\quad \mbox{ a.e. in } \left[ t_0, T \right],
\end{equation}
satisfying the transversality condition
\begin{equation}\label{dualarcStar}
-\overline{p}(T)\in \partial^{+} \phi(\overline{x}(T)).
\end{equation}
Then, $\overline{p}(\cdot)$ satisfies the full sensitivity relation 
\begin{equation}\label{NuovaInclu}
\left( H(\overline{x}(t),\overline{p}(t)),-\overline{p}(t)\right)\in  \partial^+ V(t,\overline{x}(t)) \mbox{ for all } t\in (t_0,T).
\end{equation}
%Furthermore, if $V$ is semiconcave, then \eqref{NuovaInclu} holds true for all $t\in[t_0,T].$
\end{theorem}
\begin{proof}
In view of Remark \ref{RemarkDualArc}, the above dual arc $\overline{p}(\cdot)$ satisfies the following:
\begin{enumerate}
\item[(i)] either $\overline{p}(t)\neq 0$ for all $t\in [t_0,T]$,
\item[(ii)] or $\overline{p}(t)=0$ for all $t\in [t_0,T]$.
\end{enumerate}
Suppose to be in case (i), that is $\overline{p}(t)\neq 0$ for all $t\in [t_0,T]$. Let $t\in(t_0,T)$ be fixed. Hence, $\overline{x}(\cdot)$ is the unique solution of the Cauchy problem
\begin{equation}\label{AA1}
\left\{\begin{array}{l}
\dot{x}(s)= \nabla_p H(x(s),\overline{p}(s))\ \mbox{ for all }  s\in \left[ t, T \right],\\
x(t)=\overline{x}(t).
\end{array}\right.
\end{equation}
Consider now any $(\alpha,\theta)\in\mathbb{R}\times \mathbb{R}^n$ and, for every $\tau>0$, let $x_\tau$ be the solution of the differential equation 
\begin{equation}\label{varNuova}
\left\{\begin{array}{l}
\dot{x}(s)= \nabla_p H(x(s),\overline{p}(s))\ \mbox{ for all } s\in \left[ t, T \right],\\
x(t)=\overline{x}(t)+\tau\theta.
\end{array}\right.
\end{equation}
%In view of the sublinear growth of $F$ and local Lipschitz continuity of $\nabla_p H(\cdot,p)$, we conclude that for each $r>0$ there exists a constant $K_r$ such that, for all $h \in B(0,r)$,
%\begin{equation}\label{BoooooAgain}
%\parallel x_h - \overline{x} \parallel_{\infty} \leq e^{K_r T }\mid h \mid.
%\end{equation}
By \eqref{DiniDer} and \eqref{varNuova}, we have that
\begin{equation}\label{start}
D_{\downarrow} V (t,\overline{x}(t))(\alpha,\alpha \dot{\overline{x}}(t)+ \theta)= \limsup_{\tau \rightarrow 0^+} \frac{V(t+\alpha \tau, x_\tau(t)+ \tau \alpha \dot{\overline{x}}(t)) - V(t, \overline{x}(t))}{\tau}.
\end{equation}
Moreover, from \eqref{AA1} and \eqref{varNuova},
\begin{equation}\label{ppper}
\begin{split}
\mid x_{\tau}(t+\alpha &\tau)- x_{\tau}(t) - \tau \alpha \dot{\overline{x}}(t) \mid \leq
\int_t^{t+\alpha \tau }\left\vert \nabla_p H(x_{\tau}(s),\overline{p}(s))-  \nabla_p H(\overline{x}(t),\overline{p}(t)) \right\vert  d s \\
\leq \int_t^{t+\alpha \tau } &\left\vert \nabla_p H(x_{\tau}(s),\overline{p}(s))-\nabla_p H(\overline{x}(s),\overline{p}(s))\right\vert\ ds \\
&+\int_t^{t+\alpha \tau } \left\vert \nabla_p H(\overline{x}(s),\overline{p}(s))-\nabla_p H(\overline{x}(t),\overline{p}(t))\right\vert  \ d s. 
\end{split}
\end{equation}
By \eqref{Booooo3}, \eqref{ppper}, using also that the map $x \mapsto \nabla_p H(x,p)$ is locally Lipschitz and the map $s \mapsto \nabla_p H(\overline{x}(s),\overline{p}(s))$ is continuous, we conclude that 
\begin{equation}\label{PerLip}
\mid x_{\tau}(t+\alpha \tau)- x_{\tau}(t) - \tau \alpha \dot{\overline{x}}(t) \mid = 
 o( \tau ). 
\end{equation}
Hence, from \eqref{start}, \eqref{PerLip}, using that $V$ is locally Lipschitz, the dynamic programming principle, and the transversality condition \eqref{dualarcStar} we deduce that 
\begin{equation}\label{start2}
\begin{split}
&D_{\downarrow} V (t,\overline{x}(t))(\alpha,\alpha \dot{\overline{x}}(t)+ \theta)\leq \limsup_{\tau \rightarrow 0^+} \frac{V(t+\alpha \tau, x_{\tau}(t+\alpha \tau)) - V(t, \overline{x}(t))}{\tau}\\
&\leq \limsup_{\tau \rightarrow 0^+}\frac{\phi(x_{\tau}(T)) - \phi( \overline{x}(T))}{\tau}\leq \limsup_{\tau \rightarrow 0^+}\frac{\langle - \overline{p}(T), x_{\tau}(T) - \overline{x}(T)\rangle}{\tau}.
\end{split}
\end{equation}
In view of \eqref{cambia2}, the above upper limit does not exceede $\langle - \overline{p}(t), \theta\rangle$. Recalling that $H(\overline{x}(t),\overline{p}(t))=\langle \dot{\overline{x}}(t),\overline{p}(t)\rangle$ we finally obtain 
\begin{equation}
D_{\downarrow} V (t,\overline{x}(t))(\alpha,\alpha \dot{\overline{x}}(t)+ \theta)\leq \alpha H(\overline{x}(t),\overline{p}(t))+\langle-\overline{p}(t), \alpha \dot{\overline{x}}(t)+\theta \rangle.
\end{equation}
Hence, for all $\alpha\in\mathbb{R}$ and $\theta_1\in\mathbb{R}^n$,
\begin{equation}\label{last}
D_{\downarrow} V (t,\overline{x}(t))(\alpha,\theta_1)\leq \alpha H(\overline{x}(t),\overline{p}(t))+\langle-\overline{p}(t), \theta_1 \rangle.
\end{equation}
The proof of \eqref{NuovaInclu}, in the case (i), follows from \eqref{DiniDer0}, \eqref{DiniDer}, and \eqref{last}.\\
Now suppose to be in case (ii), that is $p(s)=0$ for all $s\in [t_0,T]$. Thanks to \eqref{DiniDer0}, \eqref{DiniDer}, and the fact that $H(x,0)=0$ the inclusion \eqref{NuovaInclu} holds true if and only if, for all $(\alpha,\theta)\in\mathbb{R}\times\mathbb{R}^n$,
\begin{equation}
D_{\downarrow} V(t,\overline{x}(t))(\alpha,\theta)\leq 0\quad \mbox{ for all } t\in (t_0,T).
\end{equation}
Let $t\in(t_0,T)$ be fixed. Then, by Filippov's Theorem (see, e.g., Theorem 10.4.1 in \cite{MR1048347}), there exist constants $r, K_1, K_2$ such that, for any $0<\tau< K_1 $, the initial value problem
\begin{equation}\label{lastAga2star}
\left\lbrace
\begin{array}{l}
\dot{x}(s)\in F(x(s))\ \textit{ a.e. in } [t+\alpha \tau,T],\\
x(t+\alpha \tau)=\overline{x}(t)+\tau\theta, 
\end{array}
\right.
\end{equation}
has a solution, $x_{\tau}(\cdot)$, that satisfies the inequality
\begin{equation}\label{Booooo3star2}
\parallel x_{\tau} - \overline{x} \parallel_{\infty}\leq K_2  \tau.
\end{equation} 
Hence, from the dynamic programming principle, (\ref{Booooo3}), \eqref{dualarcStar}, and \eqref{DiniDer} we deduce that 
\begin{equation}\begin{split}
D_{\downarrow} V(t,\overline{x}(t))(\alpha,\theta)&=\limsup_{\tau \rightarrow 0^+} \frac{V(t+\alpha \tau,x_{\tau}(t+\alpha \tau))-V(t,\overline{x}(t))}{\tau}\\
&\leq \limsup_{\tau\rightarrow 0^+} \frac{\phi(x_{\tau}(T)) - \phi(\overline{x}(T))}{\tau}\leq 0.
\end{split}
\end{equation}
Then, the conclusion holds true also in case (ii).
\end{proof}
\begin{remark}
If in addition the map $\nabla_p H(\cdot,\cdot)$ is locally Lipschitz, one can show that the proximal full sensitivity relation
\begin{equation}\label{pppp}
\left( H(\overline{x}(t),\overline{p}(t)),-\overline{p}(t)\right)\in  \partial^{+,pr} V(t,\overline{x}(t)) \mbox{ for all } t\in[t_0,T]
\end{equation}
holds true when $-\overline{p}(T)\in \partial^{+,pr} \phi(\overline{x}(T))$. Note that the full sensitivity relation \eqref{pppp} implies the partial version \eqref{PPSR}. However, in Theorem~\ref{Lemma_super_nonprox} we have proved \eqref{pppp} without assuming the local Lipschitzianity of $\nabla_p H(\cdot,\cdot)$.
\end{remark}
\section{Necessary and sufficient conditions for optimality}
The first result of this section can be seen as a strengthening of the maximum principle. Roughly speaking, we want to prove the existence of a dual arc that verifies the final condition $-p(T)=q$ for any $q$ in the proximal superdifferential of the final cost. In proving a result of this kind for smooth optimal control problems (see, e.g., \cite[Theorem in 7.3.1.]{MR2041617}), a crucial role is played by the construction of a variation of control and the analysis of the behaviour of the optimal trajectory, depending on the parameters of the variation. 
%Since for this analysis the authors used the classical variational equation, 
This approach is not valid in general for differential inclusions. Here, in order to replace the variational equation, we shall introduce a further assumption on the multifunction $F$ and make use of the partial sensitivity relations proved in the previous section.
\begin{theorem}\label{NewOptimality}
Assume (SH), (H1), that $F:\mathbb{R}^n\rightrightarrows \mathbb{R}^n$ is locally strongly convex, and $\phi:\mathbb{R}^n\rightarrow\mathbb{R}$ is locally Lipschitz. Let $\overline{x}:[t_0,T]\rightarrow \mathbb{R}^n$ be an optimal solution for $\mathcal{P}(t_0,x_0)$. Then, for any $q \in \partial^{+,pr}\phi(\overline{x}(T))$, there exists a solution $\overline{p}:[t_0,T]\rightarrow \mathbb{R}^n$ of the differential inclusion
\begin{equation}\label{defP}
\left\{\begin{array}{l}
-\dot{p}(s) \in \partial_x^- H(\overline{x}(s),p(s))\quad \mbox{ a.e. in } s\in \left[ t_0, T \right], \\
-p(T)=q.
\end{array}\right.
\end{equation}
Furthermore, every such solution $\overline{p}(\cdot)$ satisfies the maximum principle
\begin{equation}\label{PRINCmax}
H(\overline{x}(t),\overline{p}(t))= \langle \dot{\overline{x}}(t),\overline{p}(t)\rangle \quad \mbox{ a.e. in } [t_0,T].
\end{equation}
If $q\neq 0$, then \eqref{PRINCmax} holds true everywhere in $[t_0,T]$.
%or, equivalently, $\overline{x}(\cdot)$ is the unique solution of Cauchy problem:
%\begin{equation}\label{Nabla}
%\left\{\begin{array}{l} 
%\dot{x}(t)=\nabla_p H (x(t),\overline{p}(t)),\ \mbox{ a.e. } t \in [t_0,T], \\
%x(t_0)=x_0.
%\end{array}\right.
%\end{equation}
\end{theorem}
\begin{proof}
Let $\overline{x}(\cdot)$ be an optimal solution for $\mathcal{P}(t_0,x_0)$. %We have:
%\begin{equation}
%-\dot{\overline{p}}(s) \in G (s,\overline{p}(s))\quad \mbox{ a.e. } s\in \left[ t_0, T \right],
%\end{equation}
%where 
Define the multifunction $G:[t_0,T]\times \mathbb{R}^n \rightrightarrows \mathbb{R}^n$ by $G(s,p)=\partial_x^- H(\overline{x}(s),p)$. Observe that:
\begin{itemize}
\item for each $(t,p)\in [t_0,T]\times \mathbb{R}^n$, $G(t,p)$ is nonempty compact and convex set;
\item by a known property of the generalized gradient, there exists a constant $k>0$ such that, $\forall (s,p)\in [t_0,T]\times \mathbb{R}^n$ and $\forall v \in G(s,p)$, it holds that $| v | \leq  k | p |$; 
%\item By a known property of the generalized gradient, $\forall (s,p)\in [t_0,T]\times \mathbb{R}^n$ and $\forall v \in G(s,p)$ it holds that $| v | \leq  k | p |$, where $k$ is the (local) Lipschitz constant of $F$ (relative to the compact containing $\overline{x}([t_0,T])$). 
\item $G$ is upper semicontinuous. 
\end{itemize}
In order to verify the last property, let us prove that $G$ has a closed graph in $[t_0,T] \times \mathbb{R}^n\times \mathbb{R}^n$. The conclusion follows because a multifunction taking values in a compact set and having a closed graph is upper semicontinuous (see e.g. Corollary 1 p. 41 in \cite{MR755330}). The graph of $G$ is
\[ 
\textit{Graph}(G) = \lbrace \left(  (t,p),q \right), (t,p)\in [t_0,T]\times \mathbb{R}^n : q \in \partial_x H(\overline{x}(t),p) \rbrace.
\]
Let $\left((t_i,p_i),q_i \right)$ be a sequence in $Graph(G)$ which converges to some $((t,p),q)$. Thus, by $q_i \in \partial_x^- H(\overline{x}(t_i),p_i)$, there exists an open set $A$ containing $\overline{x}([t_0,T])$ and a constant $c=c(A)$ so that
\begin{equation}\label{ClosGra}
H(y,p_i) - H(\overline{x}(t_i),p_i) - \langle q_i, y - \overline{x}(t_i) \rangle  \geq - c |p_i| |y - \overline{x}(t_i)|^2 
\end{equation}
for any point $y \in A$ and $i$ large enough. Passing to the limit in (\ref{ClosGra}), we see that $q \in \partial_x^- H(\overline{x}(t),p)$. This proves that the graph of $G$ is closed. \\
%\item the multifunction $s \rightarrow G(s,p)$ is measurable for any fixed $p\in \mathbb{R}^n$. Indeed, by a well-known property of the generalized gradient (see for instance Proposition $10.10$ in \cite{C}), the multifunction $y \rightarrow \partial_x H(y,p)$ is upper semicontinuous, which implis that it is Borel measurable (see for istance ). Thus for any Lebesgue (Borel) measurable function $\gamma(\cdot)$ the multifunction $s \rightarrow \partial_x H(\gamma(s),p)$ is Lebesgue (Borel) measurable. 
%\item we shall prove that it is also upper semicontinuous. Let us prove that its graph is closed, then the conclusion follows by Corollary $1$ at page. $42$ in \cite{A}. Let $(p_i,q_i)$ be in the graph of the multifunction, that means $q_i \in \partial_x H(x,p_i)$. Due to the local semiconcexity of the function $x \rightarrow H(x,p)$, there exists a compact set $A$ so that
%\begin{equation}\label{ClosGra}
%H(y,p_i) - H(x,p_i) - \langle q_i, y - x \rangle  \leq c |y - x|^2 
%\end{equation}
%for any point $y \in A$ such that $[y, x_i] \subset A$. If successions $p_i$ and $q_i$ converg to $p$ and $q$ respectively, then passing to the limit in (\ref{ClosGra}) it implies that $q \in \partial_x H(x,p)$.

We deduce from the above three properties the existence of at least one solution $\overline{p}(\cdot)$ of (\ref{defP}) on $[t_0,T]$.\\
Now let us investigate the equality \eqref{PRINCmax}. Consider, first, the case where $q=0$. Then, thanks to Remark \ref{RemarkDualArc}, we conclude that the solution of (\ref{defP}) vanishes on $[t_0,T]$ and so the equality \eqref{PRINCmax} is obvious. Let now $\overline{p}(\cdot)$ be any solution of (\ref{defP}) for some $q \in \partial^{+,pr}\phi(\overline{x}(T))\smallsetminus\lbrace 0\rbrace$. In this case, since Remark \ref{RemarkDualArc} ensures that $\overline{p}(\cdot)$ never vanishes on $[t_0,T]$, recalling Remark~\ref{RemarkContinuity}, we conclude that our claim \eqref{PRINCmax} is equivalent to the identity
\begin{equation} 
\dot{\overline{x}}(t)=\nabla_p H (\overline{x}(t),\overline{p}(t)) \mbox{ for all } t \in [t_0,T].
\end{equation}
Now let $\tau$ be such that $0<\tau<T-t_0$ and define an admissible trajectory $x:[t_0,T]\rightarrow \mathbb{R}^n$ in the following way:
\begin{itemize}
\item on the interval $[t_0,T-\tau)$, $x(\cdot)$ coincides with the optimal trajectory $\overline{x}(\cdot)$, 
\item on the interval $[T-\tau,T]$, $x(\cdot)$ is the solution of the Cauchy problem:
\begin{equation}\label{defVariaz}
\left\{\begin{array}{l} 
\dot{x}(t)= \nabla_p H(x(t),\overline{p}(t)) \ \mbox{ for all } t \in [T-\tau,T], \\
x(T-\tau)=\overline{x}(T-\tau).
\end{array}\right.
\end{equation}
\end{itemize}
We are going to give a first estimate of $\|\overline{x}-x\|_{\infty}$. By \eqref{defVariaz} and the local Lipschitzianity of the map $\nabla_p H(\cdot,p)$, there exist $0<\tau_1<T-t_0$ and $c \geq 0$ such that, for every $\tau<\tau_1$ and $t\in [T-\tau,T]$,
\begin{equation}
\begin{split}
&|\overline{x}(t)-x(t)|\leq \int_{T-\tau}^t |\dot{\overline{x}}(s)-\nabla_p H(x(s),\overline{p}(s)) |\ d s \\
&\leq \int_{T-\tau}^t \left(|\nabla_p H(\overline{x}(s),\overline{p}(s))-\nabla_p H(x(s),\overline{p}(s)) | +|\dot{\overline{x}}(s) - \nabla_p H(\overline{x}(s),\overline{p}(s))| \right) ds  \\
&\leq c \int_{T-\tau}^t |\overline{x}(s)-x(s)| ds + \int_{T-\tau}^t |\dot{\overline{x}}(s) - \nabla_p H(\overline{x}(s),\overline{p}(s))| ds.
\end{split}
\end{equation}
By the Gronwall inequality, the above estimate yields that, for every $t\in [T-\tau,T]$, 
\begin{equation}\label{stimaGron}
|\overline{x}(t)-x(t)| \leq e^{c \tau}  \int_{T-\tau}^T |\dot{\overline{x}}(s) - \nabla_p H(\overline{x}(s),\overline{p}(s))| ds.
\end{equation}
The next point is to find a good estimate for the integral on the right side of (\ref{stimaGron}). Since $-\overline{p}(T)\in \partial^{+,pr}\phi(\overline{x}(T))$, there exist constant $c_1, r$ such that, if $x(T)\in B(\overline{x}(T),r)$, it holds that
\begin{equation}\label{proxStim}
\phi(x(T))-\phi(\overline{x}(T))+\langle \overline{p}(T),x(T)-\overline{x}(T) \rangle \leq c_1 |x(T)-\overline{x}(T)|^2.
\end{equation}
From (\ref{stimaGron}) we deduce that if $\tau<\min \lbrace \tau_2,\tau_1 \rbrace$, where
$$
\tau_2 := \frac{1}{c} ln \left( \frac{r}{ \parallel \dot{\overline{x}}\parallel_{\mathcal{L}^1([t_0,T])}+ \parallel  \nabla_p H (\overline{x},\overline{p}) \parallel_{\mathcal{L}^1([t_0,T])}}\right),
$$
then $x(T)\in B(\overline{x}(T),r)$ and so \eqref{proxStim} is true. Hence, by the optimality of $\overline{x}(\cdot)$ and (\ref{proxStim}), it follows that
\begin{equation}\label{ProxStimDopo}
\langle \overline{p}(T),x(T)-\overline{x}(T) \rangle \leq c_1 | x(T)-\overline{x}(T)|^2.
\end{equation}
Furthermore, since $x(T-\tau)=\overline{x}(T-\tau)$, we have that
\begin{equation}\label{11}
\begin{split}
\langle \overline{p}(T),x(T)-\overline{x}(T) \rangle &=  \int_{T-\tau}^T  \langle \dot{\overline{p}}(s), x(s)-\overline{x}(s) \rangle\ ds +\int_{T-\tau }^T \langle \overline{p}(s),\dot{x}(s)-\dot{\overline{x}}(s) \rangle \ d s \\
&= (I) + (II)
\end{split}
\end{equation}
We can estimate the first term using the assumption (H1)(i):
\begin{equation}\label{22}
(I)\geq \int_{T-\tau}^T \left( H(\overline{x}(s),\overline{p}(s))-H(x(s),\overline{p}(s))- c_2 |x(s)-\overline{x}(s)|^2\right) ds,
\end{equation}
where $c_2$ is a suitable constant. The second term can be estimated using  Proposition \ref{equiv}, 5. and the identity \eqref{derive}, to obtain 
\begin{equation}\label{33}
\begin{split}
&(II)=\int_{T-\tau}^{T}\left( H(x(s),\overline{p}(s))-  H(\overline{x}(s),\overline{p}(s))+ \langle \overline{p}(s), \nabla_p H(\overline{x}(s),\overline{p}(s))- \dot{\overline{x}}(s) \rangle \right) \\
&\geq \int_{T-\tau}^{T} \left( H(x(s),\overline{p}(s))-  H(\overline{x}(s),\overline{p}(s)) +c_3 |\overline{p}(s)||\nabla_p H(\overline{x}(s),\overline{p}(s))- \dot{\overline{x}}(s)|^{2} \right) ds,
\end{split}
\end{equation}
for some constant $c_3$. The four previous estimates together imply that there exist constants $c_{4},c_5\geq 0$ such that
\begin{equation}\label{tecni}
\begin{split}
\int_{T-\tau}^T |\nabla_p H(\overline{x}(s),\overline{p}(s))-\dot{\overline{x}}(s)|^{2}\ ds &\leq c_4 \int_{T-\tau}^T |\overline{x}(s)-x(s)|^2\ ds + c_5 |\overline{x}(T)-x(T)|^2\\
&\leq(c_4 \tau +c_5)\parallel x - \overline{x} \parallel_{\infty}^2.
\end{split}
\end{equation}
Now we can go back to (\ref{stimaGron}) and estimate the integral on the right side using (\ref{tecni}) and the H\"older inequality. We obtain that
%\begin{equation}
%|\overline{x}(t)-x(t)|\leq e^{c \tau}\tau^{\frac{1}{2}} \left( c_4 \tau +c_5 \right)^{\frac{1}{2}} \parallel \overline{x}-x \parallel_{\infty},\  \forall t \in [t_0,T],
%\end{equation}
%and, therefore,
\begin{equation}
\parallel \overline{x}-x \parallel_{\infty} \leq e^{c \tau}\tau ^{\frac{1}{2}} \left( c_4 \tau + c_5 \right)^{\frac{1}{2}} \parallel \overline{x}-x \parallel_{\infty}.
\end{equation}
Let $\tau_3$ be such that $e^{2 c \tau_3}\tau_3 (c_4 \tau_3 +c_5)=1$. Then, choosing $\tau < \min_{i=1,2,3}\lbrace \tau_i \rbrace$, it follows that $\parallel \overline{x}-x\parallel_{\infty}=0$ and so, by (\ref{tecni}), we finally conclude that
\begin{equation}\label{quasiTesi}
\dot{\overline{x}}(t)=\nabla_p H (\overline{x}(t),\overline{p}(t)) \mbox{  a.e. in } t \in [T-\tau,T].
\end{equation}
Since the arc $\overline{p}(\cdot)$ never vanishes, the pair $(\overline{x}(\cdot),\overline{p}(\cdot))$ stays in the set where $H$ is differentiable with respect to $p$. So, \eqref{quasiTesi} is true for all $t\in [T-\tau,T]$. In order to show that the equality in \eqref{quasiTesi} holds on the whole interval $[t_0,T]$, one can iterate the above argument, making use of Theorem \ref{Lemma_super_nonprox}, and reach the conclusion in a finite number of steps. Let us only sketch the second step. Recall that the trajectory $\overline{x}|_{[T-\tau,T]}:[T-\tau, T]\rightarrow \mathbb{R}^n$ is optimal for the Mayer problem  $\mathcal{P}( T-\tau, \overline{x}(T-\tau))$ and we have just shown that $\overline{x}(\cdot)$ is the solution of the Cauchy problem
\begin{equation}
\left\{\begin{array}{l} 
\dot{x}(t)=\nabla_p H (x(t),\overline{p}(t)),\quad t \in [T-\tau,T], \\
x(T-\tau)= \overline{x}(T-\tau).
\end{array}\right.
\end{equation}
Hence, by Theorem \ref{Lemma_super_nonprox}, there exist constants $k, r_1 >0$ such that, for all $h\in B(0,r_1)$, it holds that
\begin{equation}\label{proxStim2}
V(T-\tau, \overline{x}(T-\tau)+h)- V(T-\tau, \overline{x}(T-\tau))+\langle \overline{p}(T-\tau),h \rangle \leq  k |h|^2.
\end{equation}
We define a trajectory $x:[t_0,T-\beta]\rightarrow \mathbb{R}^n$, with $\tau<\beta<T-t_0$, which on the interval $[t_0,T-\beta)$ coincides with $\overline{x}(\cdot)$, and on the interval $[T-\beta,T-\tau]$ is the solution of the problem
\begin{equation}
\left\{\begin{array}{l} 
\dot{x}(t)= \nabla_p H(x(t),\overline{p}(t)) ,\quad t \in [T-\beta,T-\tau], \\
x(T-\beta)=\overline{x}(T-\beta).
\end{array}\right.
\end{equation}
Choosing $\beta$ such that $x(T-\tau)\in B(\overline{x}(T-\tau),r_1)$, estimate (\ref{proxStim2}) yields
\begin{equation}\label{proxStim4}
\begin{split}
V(T-&\tau, x(T-\tau))- V(T-\tau, \overline{x}(T-\tau))+\langle \overline{p}(T-\tau), x(T-\tau)-\overline{x}(T-\tau) \rangle \\
&\leq  k |x(T-\tau)-\overline{x}(T-\tau)|^2.
\end{split}
\end{equation}
Thus, using the dynamic programming principle, we obtain that
\begin{equation}\label{proxStim3}
\begin{split}
V(T-\tau,&x(T-\tau)) - V( T-\tau,\overline{x}(T-\tau))\\
&\geq V(T-\beta,x(T-\beta))-V(T-\beta,\overline{x}(T-\beta))=0.
\end{split}
\end{equation}
Then, from (\ref{proxStim4}) and (\ref{proxStim3}), we deduce that
\begin{equation}
\langle \overline{p}(T-\tau),x(T-\tau)-\overline{x}(T-\tau) \rangle \leq k |x(T-\tau)-\overline{x}(T-\tau)|^2,
\end{equation}
which replaces (\ref{ProxStimDopo}). The estimates right after (\ref{ProxStimDopo}) can be easily adapted to imply, finally, that the equality in (\ref{quasiTesi}) holds on the interval $[T-\beta,T]$, for some suitable $\tau < \beta\leq T-t_0$. The conclusion on $[t_0,T]$ can be reached in a finite number of steps, since the constant $r$ such that (\ref{iclu_super}) holds true for all $h \in  B(0,r)$ are independent from $t\in [t_0,T]$.
\end{proof}
\begin{remark}
The above theorem and Theorem \ref{Lemma_super_nonprox} give together that for any $q\in \partial^{+,pr}\phi(\overline{x}(T))$ and for any solution $\overline{p}$ of \eqref{defP} both the maximum principle and the full sensitivity relation \eqref{NuovaInclu} hold true. This is a less restrictive conclusion that the one of Theorem \ref{TheoDualArc2}, which only affirms that for some $q\in\partial\phi(\overline{x}(T))$ and some solution $\overline{p}(\cdot)$ the maximum principle \eqref{MAXIMUMprincipleEQUATION} holds true. For this reason, using the proximal superdifferential of $\phi$ instead of the generalized gradient seems to be more appropriate whenever (H1) is satisfied and $F$ is locally strongly convex. 
\end{remark}
Now we are ready to give a set of necessary and sufficient conditions for optimality.  
\begin{theorem}
Assume (SH) and (H1). Let $\phi: \mathbb{R}^n \rightarrow \mathbb{R}$ be locally semiconcave and $F:\mathbb{R}^n\rightrightarrows \mathbb{R}^n$ locally strongly convex. A solution $\overline{x}:[t_0,T]\rightarrow \mathbb{R}^n$ of the system \eqref{May1}-\eqref{May2} is optimal for $\mathcal{P}(t_0,x_0)$ if and only if, for every $q \in \partial^{+} \phi(\overline{x}(T))$, any solution $\overline{p}:[t_0,T]\rightarrow \mathbb{R}^n$ of the differential inclusion
\begin{equation}\label{DualArc4}
-\dot{p}(t)\in \partial_x^- H(\overline{x}(t),p(t))\quad   \mbox{ a.e. in } \left[ t_0, T \right]
\end{equation}
with the transversality condition 
\begin{equation}\label{dualarcStar4}
-\overline{p}(T)=q, 
\end{equation}
satisfies the full sensitivity relation
\begin{equation}\label{NuovaInclu4}
\left( H(\overline{x}(t),\overline{p}(t)),-\overline{p}(t)\right)\in  \partial^+ V(t,\overline{x}(t)) \quad \textit{ for all } t\in(t_0,T),
\end{equation}
and the maximum principle
\begin{equation}\label{MaxPrinc4}
H(\overline{x}(t),\overline{p}(t))= \langle \overline{p}(t), \dot{\overline{x}}(t) \rangle\quad \mbox{ a.e. in } [t_0,T].
\end{equation}
%Furthermore, if $V$ is semiconcave, then \eqref{NuovaInclu} holds true for all $t\in[t_0,T].$
\end{theorem}
\begin{proof}
The sufficiency follows from Theorem \ref{Sufficient}. The fact that the existence of an arc $p(\cdot)$ satisfying \eqref{dualarcStar4} and \eqref{MaxPrinc4} is a necessary condition for optimality consists in Theorem \ref{NewOptimality}. We need just to recall that if $\phi$ is a locally semiconcave function, then $\partial^{+,pr}\phi(\overline{x}(T))=\partial^{+}\phi(\overline{x}(T))$ and this set is nonempty. Finally, the full sensitivity relation \eqref{NuovaInclu4} comes from Theorem \ref{FullSensivTheorem}.
\end{proof}
\section{Relations between reachable gradients of the value function and optimal trajectories}
In the calculus of variations, the existence of a one-to-one correspondence between the set of minimizers starting from a point $(t,x)$ and the reachable gradient of the value function at $(t,x)$ is a well-known fact. This allows, among other things, to identify the set of singular points of the value function as the set of starting points for more than one minimizer. The aim of this last section is to investigate the relations that occur in our context. Here, a difficulty consists in the singularity of the Hamiltonian at $p=0$, that forces us to study separately the case where $0 \in \partial^{*}V(t,x)$. When the Hamiltonian and the endpoint cost are in the class $C^{1,1}_{loc}( \mathbb{R}^n\times (\mathbb{R}^n\smallsetminus \lbrace 0 \rbrace))$ and $C^1(\mathbb{R}^n)$, respectively, there is an injective map from $\partial^{*}V(t,x)\smallsetminus \lbrace 0 \rbrace$ to the set of all optimal trajectories starting from $(t,x)$. See for instance Theorem 7.3.10 in \cite{MR2041617}, where the authors give the proof for optimal control problems with smooth data. However, we assume here neither the existence of a smooth parameterization, nor such a regularity of the Hamiltonian. It is precisely the lack of regularity of $H$ that represents the main difficulty, since it does not guarantee the uniqueness of solutions of the system \eqref{NewHamilt} below. We shall prove that in our case, under suitable assumptions, there exists an injective set-valued map from $\partial^{*}V(t,x)\smallsetminus \lbrace 0 \rbrace$ into the set of optimal trajectories starting from $(t,x)$.
\begin{lemma}\label{StranoBefore}
Let (SH), (H1) hold and $\phi\in C^1(\mathbb{R}^n)$. Given a point $(t,x)\in [0,T]\times \mathbb{R}^n$ and a vector $\overline{p}=(\overline{p}_t,\overline{p}_x)\in \partial^{*}V(t,x)\smallsetminus \lbrace 0 \rbrace$, there exists at least one pair $(y(\cdot),p(\cdot))$ that satisfies the system 
\begin{equation}\label{NewHamilt}
\left\lbrace
\begin{array}{rlll}
\dot{y}(s)&= &\nabla_p H(y(s),p(s)) & \mbox{ for all }s\in [t,T],  \\
-\dot{p}(s)&\in &\partial_x^- H(y(s),p(s))&\mbox{ a.e. in }[t,T],
\end{array}\right.
\end{equation}
and initial conditions 
\begin{equation}\label{Condition}
\left\lbrace
\begin{array}{l}
y(t)=x,  \\
p(t)=-\overline{p}_x,
\end{array}\right.
\end{equation}
such that $y(\cdot)$ is optimal for $\mathcal{P}(t,x)$.  
\end{lemma}
\begin{proof}
Observe, first, that if $(\overline{p}_t,\overline{p}_x)\in \partial^{*}V(t,x)\smallsetminus \lbrace 0 \rbrace$, then $\overline{p}_x \neq 0$. Indeed, every point $(\overline{p}_t,\overline{p}_x)\in \partial^{*}V(t,x)$ satisfies the equation $-p_t + H(x,-p_x)=0$. Furthermore, $H(x,0)=0$. Consequently, if $(p_t,p_x) \neq 0$, then $p_x \neq 0$.\\
Since $(\overline{p}_t,\overline{p}_x)\in \partial^{*}V(t,x)$, we can find a sequence $\{(t_k,x_k)\}$ such that $V$ is differentiable at ${(t_k,x_k)}$ and
\[
\lim_{k\rightarrow \infty} (t_k,x_k)=(t,x), \quad \lim_{k\rightarrow \infty} \nabla_x V(t_k,x_k)=-\overline{p}_x.
\]
Let $y_k(\cdot)$ be an optimal trajectory for $\mathcal{P}(t_k,x_k)$. Let us prove, first, that $\nabla \phi (y_k(T))\neq 0$ for $k$ large enough. Since $\overline{p}_x \neq 0$, we have that there exists $\overline{k}>0$ such that $\nabla_x V(t_k,x_k)\neq 0$ for $k>\overline{k}$. Now fix $k>\overline{k}$. Then, there exists $\theta \in \mathbb{R}^{n}$ such that $\langle \nabla_x V (t_k,x_k),\theta \rangle>0$. Take a sequence $s_i\rightarrow 0^+$ and consider the problem 
\begin{equation}\label{Ahah}
\left\{\begin{array}{l}
\dot{z_i}(s) \in F(z_i(s)) \ \mbox{ a.e. in } \left[ t_k, T \right], \\
z_i(t_k)= y_k(t_k)+s_i \theta. 
\end{array}\right.
\end{equation}
By Filippov's Theorem (see, e.g., \cite[Theorem 10.4.1]{MR1048347}), there exists a solution $z_i(\cdot)$ of (\ref{Ahah}) such that  $\parallel z_i - y_k \parallel_{\infty}\leq c s_i$, for some $c>0$ independent of $i$. Hence, when $i \rightarrow \infty$,
\begin{equation}\label{Ahah2}
\frac{V(t_k,y_k(t_k)+s_i \theta)-V(t_k,y_k(t_k))}{s_i} \rightarrow \langle \nabla_x V (t_k,x_k),\theta \rangle > 0. 
\end{equation}
Moreover, by the dynamic programming principle, \eqref{Ahah}, and \eqref{Ahah2} we obtain that
\begin{equation}\label{DerNonNull}
\begin{split}
&\limsup_{i \rightarrow \infty} \frac{\phi(z_i(T))-\phi(y_k(T))}{s_i} \geq 
\lim_{i \rightarrow \infty} \frac{V(t_k,z_i(t_k))-V(t_k,y_k(t_k))}{s_i} 
\\
&=\lim_{i \rightarrow \infty} \frac{V(t_k,y_k(t_k)+s_i \theta)-V(t_k,y_k(t_k))}{s_i}=\langle \nabla_x V (t_k,x_k),\theta \rangle > 0. 
\end{split}
\end{equation}
Now we can consider a subsequence $\lbrace z_{i_j} \rbrace$ and a $\gamma\in\mathbb{R}^n$ such that
\[
\frac{z_{i_j}(T)-y_k(T)}{s_{i_j}}\rightarrow \gamma.
\]
Then, by (\ref{DerNonNull}), it follows that 
\[
\langle \nabla \phi (y_k(T)),\gamma \rangle > 0.
\]
We can conclude that $\nabla \phi (y_k(T))\neq 0$. This condition allows to say that there exists a nonvanishing arc $p_k (\cdot)$ such that, for each $k$ large enough, the pair $(y_k(\cdot),p_k(\cdot))$ solves the system \eqref{NewHamilt} and $-p_k (T) = \nabla \phi (y_k(T))$. We note that $-p_k (T)$ is an element in the set $\partial^+ \phi (y_k(T))$ and so the sensitivity relation described in Remark \ref{RelarkPartialSensRelation} holds. This means that
\begin{equation}\label{gradAppro}
-p_k(s)\in \partial^+_x V(s,y_k(s))\quad \forall s \in [t_k,T].
\end{equation}
Moreover, recalling that $V$ is differentiable at ${(t_k,x_k)}$, it holds that $-p_k (t_k)=\nabla_x V(t_k,x_k)$ and so the pair $(y_k(\cdot),p_k (\cdot))$ satisfies the initial conditions $y_k(t_k)=x_k, -p_k (t_k)=\nabla_x V(t_k,x_k)$.\\ 
The last argument consists in proving that the sequence $(y_k(\cdot),p_k (\cdot))$, after possibly passing to a subsequence, converges to a pair $(y(\cdot),p(\cdot))$ that verifies our claims. It is easy to prove that the sequences of functions $\lbrace p_k \rbrace_k$ and $\lbrace y_k\rbrace_k$ are uniformly bounded and uniformly Lipschitz continuous in $[t,T]$, using Gronwall's inequality together with estimates (\ref{Gronwall}) and (SH)(iii), respectively. Hence, after possibly passing to subsequences, we may assume that the sequence $(y_k(\cdot),p_k (\cdot))$ converges uniformly in $[t,T]$ to some pair of Lipschitz functions $(y(\cdot),p(\cdot))$. Moreover, $\dot{p}_k (\cdot)$ converges weakly to $\dot{p}(\cdot)$ in $L^1([t,T];\mathbb{R}^n)$. Furthermore, we can say that 
$$((y_k(s),p_k (s)),-\dot{p}_k(s))\in Graph(M)\quad \mbox{ a.e. in } [t_k,T],$$ 
where $M$ is given by $M(x,p)=\partial_x^- H(x,p)$. The multifunction $M$ is upper semicontinuous on its domain; this can be easily derived as done for $G$ in the proof of Theorem \ref{NewOptimality}. Hence, from Theorem 7.2.2. in \cite{MR1048347} it follows that $-\dot{p}(s)\in \partial_x^- H(y(s),p(s))$ for a.e. $s \in [t,T]$. By the continuous differentiability of $\phi$, we get that $-p(T)=\nabla \phi(y(T))\neq 0$ and so, recalling Remark \ref{RemarkDualArc}, the arc $p(\cdot)$ never vanishes. Hence, recalling also that the map $(x,p)\mapsto \nabla_p H(x,p)$ is continuous for $p\neq 0$, we easily get that $\dot{y}(s)=\nabla_p H(y(s),p(s))$ for all $s\in[t,T]$. In conclusion, $(y(\cdot),p(\cdot))$ is a solution of the adjoint system \eqref{NewHamilt} with initial conditions $y(t)=x,\ p(t)=-\overline{p}_x \neq 0$. This implies of course that $\dot{y}(s)\in F(y(s))$ for all $s\in [t,T]$. Moreover, since $V$ is continuous and $y_k(\cdot)$ are optimal, we have
\[
\phi (y(T))= \lim_{k\rightarrow \infty} \phi (y_k(T))= \lim_{k\rightarrow \infty} V(t,y_k(t)) = V(t,y(t)),
\]
which means that $y(\cdot)$ is optimal for $\mathcal{P}(t,x)$.
\end{proof}
For any $\overline{p}=(\overline{p}_t,\overline{p}_x)\in \partial^{*}V(t,x)\smallsetminus \lbrace 0 \rbrace$, we denote by $\mathcal{R}(\overline{p})$ the set of all trajectories $y(\cdot)$ that are solution of \eqref{NewHamilt}-\eqref{Condition}, and optimal for $\mathcal{P}(t,x)$. The above theorem guarantees that the set-valued map $\mathcal{R}$ that associates with any $p \in \partial^{*}V(t,x)\smallsetminus \lbrace 0 \rbrace$ the set $\mathcal{R}(p)$ has nonempty values. Now let us prove that $\mathcal{R}$ is strongly injective. We will use the ``difference set'':
\[
\partial_x^- H(x,p)-\partial_x^- H(x,p):=\lbrace a - b :\ a, b \in \partial_x^- H(x,p)\rbrace.
\]
\begin{theorem}\label{Strano}
Under all the hypotheses of Lemma \ref{StranoBefore}, if we assume in addition that
\begin{description}
\item[(H3)]  
for each $x\in \mathbb{R}^n$, $F(x)$ is not a singleton and, if $n > 1$, it has a $C^1$ boundary, 
\item[(H4)] 
$\mathbb{R}^+ p \cap \left( \partial_x^- H(x,p)- \partial_x^- H(x,p)\right) = \emptyset \quad \forall p\neq 0,$
\end{description}
then for any $\overline{p}_{1},\overline{p}_{2} \in \partial^{*}V(t,x)\smallsetminus \lbrace 0 \rbrace$ with $\overline{p}_{1}\neq \overline{p}_{2}$, we have that $\mathcal{R}(\overline{p}_1)\cap \mathcal{R}(\overline{p}_2)=\emptyset$. Equivalently, the set-valued map $\mathcal{R}$ is strongly injective.
\end{theorem} 
\begin{proof}
Suppose to have two elements $\overline{p}_i=(\overline{p}_{i,t},\overline{p}_{i,x})\in \partial^{*}V(t,x)\smallsetminus \lbrace 0 \rbrace,\ i=1,2$, with $\overline{p}_1\neq \overline{p}_2$. Note that the Hamilton-Jacobi equation implies that $\overline{p}_{1,x}\neq \overline{p}_{2,x}$ if and only $\overline{p}_1\neq \overline{p}_2$. Furthermore, suppose that there exist two pairs $(y(\cdot),p_{i}(\cdot)),\ i=1,2$ that are solutions of the system \eqref{NewHamilt} with $p_i(t)=\overline{p}_{i,x}$ and $y(\cdot)$ is optimal for $\mathcal{P}(t,x)$. Then, we get
\[
\dot{y}(s)=\nabla_p H(y(s),p_1(s))=\nabla_p H(y(s),p_2(s))\quad \mbox{ for all } s\in [t,T].
\]
This implies that
$$p_{i}(s) \in N_{F(y(s)) }\left(\dot{y}(s) \right),\ i=1,2\ \mbox{ for all } s\in [t,T]. $$
By (H3), the normal cone $N_{F(y(s)) }\left(\dot{y}(s) \right)$ is a half-line. Recalling also that $p_{i},i=1,2$ never vanish, it follows that there exists $\lambda(s)> 0$ such that $p_2(s)=\lambda(s)p_1(s)$, for every $s\in [t,T]$. The function $\lambda(\cdot)$ is differentiable a.e. on $[t,T]$ because 
\[
\lambda(s)=\frac{\mid p_2(s)\mid }{\mid p_1(s)\mid }.
\]
By \eqref{NewHamilt} and since $\beta \partial_x^- H(x,p)= \partial_x^- H( x,\beta p)$ for each $\beta>0$, it follows that
$$
-\dot{p}_2(s)= -\lambda(s)\dot{p}_1(s) -\dot{\lambda}(s)p_1(s)\in \lambda(s) \partial_x^- H(y(s), p_1(s))\quad \mbox{ for a.e. } s\in [t,T].
$$
Dividing by $\lambda(s)$,
\begin{equation}\label{aggiunta}
-\frac{\dot{\lambda}(s)}{\lambda(s)} p_1(s)\in \partial_x^- H(y(s), p_1(s))- \partial_x^- H(y(s), p_1(s))\quad \mbox{ for a.e. } s\in [t,T].
\end{equation}
Since the ``difference set'' in (H4) is symmetric, (H4) is equivalent to the condition
\begin{equation}\label{H4Nuov}
(\mathbb{R}\smallsetminus \lbrace 0 \rbrace ) \ p \cap \left( \partial_x^- H(x,p)- \partial_x^- H(x,p)\right) = \emptyset \quad \forall p\neq 0.
\end{equation}
From \eqref{aggiunta} and \eqref{H4Nuov} we obtain that $\dot{\lambda}(s)= 0$ a.e. in $[t,T]$. This implies that $\lambda$ is constant and since $\nabla \phi(y(T))=-p_i(T),\ i=1,2$, this constant must be $\lambda=1$. But this yelds $\overline{p}_{1,x}=\overline{p}_{2,x}$, which contradicts the inequality $\overline{p}_1\neq \overline{p}_2$. Hence, we can assert that the functions $y_i(\cdot)$, $i=1, 2,$ are different.
\end{proof}
\begin{remark}
Recall that the assumption $\phi\in C^1(\mathbb{R}^n)$ is essential; see, e.g., Example 7.2.10 in \cite{MR2041617} where the value function is singular at points from which an unique optimal solution starts. Concerning (H4), note that it is verified, for instance, when $x \mapsto H(x,p)$ is differentiable, without any Lipschitz regularity of the map $x\mapsto \nabla_p H(x,p)$.
\end{remark}
Now let us consider the case when $\overline{p}=0\in\partial^{*}V(t,x)$.
\begin{theorem}\label{ZeroTra}
Assume (SH), (H1) and that $\phi\in C^1(\mathbb{R}^n)$, and let $(t,x)\in [t_0,T[\times \mathbb{R}^n$ be such that $0\in \partial^{*}V(t,x)$. Then there exists an optimal trajectory $y:[t,T]\rightarrow \mathbb{R}^n$ for $\mathcal{P}(t,x)$ such that $\nabla \phi (y(T))=0$. Consequently, the unique corresponding dual arc is equal to zero. 
\end{theorem} 
\begin{proof}
Since $0\in \partial^{*}V(t,x)$, we can find a sequence $\{(t_k,x_k)\}$ such that $V$ is differentiable at ${(t_k,x_k)}$ and
\[
\lim_{k\rightarrow \infty} (t_k,x_k)=(t,x), \ \lim_{k\rightarrow \infty} \nabla V(t_k,x_k)=0.
\]
Let $y_k(\cdot)$ be an optimal trajectory for $\mathcal{P}(t_k,x_k)$ and $p_k(\cdot)$ be a dual arc. By Theorem $7.2.2.$ in \cite{MR1048347}, we can assume, after possibly passing to a subsequence, that $y_k(\cdot)$ converges uniformly to $y(\cdot)$ which is a trajectory of our system. Since
\[
\phi (y(T))= 
\lim_{k\rightarrow \infty} \phi (y_k(T))= \lim_{k\rightarrow \infty} V(t_k,x_k)=V(t,x),
\]
it follows that $y$ is optimal for $\mathcal{P}(t,x)$. 
%Note that it may happen that the dual arc $p_k(\cdot)$ is the null arc for some $k$. Then of course $p_k(T)=0$. If not, recall that $-p_k(T)= \nabla \phi(y_k(T))$. So 
Furthermore, the sensitivity relation in Remark \ref{RelarkPartialSensRelation} holds true and so, recalling that $V$ is differentiable at $(t_k,x_k)$, we have 
\begin{equation}\label{Difzero}
-p_k(t_k)=\nabla_x V(t_k,x_k)\rightarrow 0.
\end{equation}
By \eqref{Difzero}, \eqref{Gronwall} and the Gronwall's inequality, we get that $p_k(T)\rightarrow 0$ when $k \rightarrow \infty$. We conclude that
\[
\nabla \phi (y(T))= \lim_{k\rightarrow \infty} \nabla \phi (y_k(T))= \lim_{k\rightarrow \infty} -p_k(T)=0.
\]
\end{proof}
\begin{remark}
The conclusion of this theorem is weaker than the previous one; it gives only the existence of an optimal trajectory without saying that $y(\cdot)$ is the solution of a system like \eqref{NewHamilt}, because in this case the dual arc associated to $y(\cdot)$ must be null everywhere. 
%Statements of this type are well-known for Mayer's problem in an optimal control problem with Hamiltonian in the class $C^{1,1}_{loc}$ (see  e.g. Theorem 7.3.11 in \cite{MR2041617}).
\end{remark}
\begin{corollary}\label{Ok0}
Assume (SH), (H1), $\phi \in C^1(\mathbb{R}^n)$ and suppose also that $\nabla \phi (x)\neq 0$ for all $x\in\mathbb{R}^n$. Then $0 \not\in \partial^{*}V(t,x)$ for all $(t,x)\in [t_0,T]\times \mathbb{R}^n$.
\end{corollary}
%Now it is possible to show that the value function is differentiable at every regular.
%\begin{corollary}\label{Ok1}
%Assume (SH), (H1), (H3), (H4). Moreover, suppose that $\phi\in C^{1}(\mathbb{R}^n) \cap SC (\mathbb{R}^n)$ and $\nabla \phi (x)\neq 0$ for all $x\in\mathbb{R}^n$. If the optimal solution of $\mathcal{P}(t,x)$ is unique, then $V$ is differentiable in $(t,x)$.
%\end{corollary}
%\begin{proof}
%Recall that $V$ is semiconcave (see \cite{MR2728465}) and therefore $V$ is differentiable at $(t,x)$ if and only if $\partial^{*}V(t,x)$ is a singleton. Thus, the result follows from Theorem \ref{Strano} and Corollary \ref{Ok0}.
%\end{proof}
The previous theorems imply that, if $(t,x)$ is a singular point of $V$, then $\mathcal{P}(t,x)$ admits more than one optimal trajectory, as proved in the following proposition. 
\begin{proposition}
Assume (SH), (H1), (H3), (H4) and let $\phi\in C^{1}(\mathbb{R}^n) \cap SC (\mathbb{R}^n)$. If $V$ fails to be differentiable at a point $(t,x)$, then there exist two or more optimal trajectory starting from $(t,x)$.
\end{proposition}
\begin{proof}
Since $V$ is semiconcave (see \cite{MR2728465}), if it is not differentiable at a point $(t,x)\in [t_0,T]\times\mathbb{R}^n$, then we can find two distinct elements $p_{1},p_{2} \in \partial^{*}V(t,x)$. If $p_{1},p_{2}$ are both nonzero, we can apply Theorem \ref{Strano} to find two distinct optimal trajectories. If one of the two vectors is zero, for instance $p_{1}$, then there exists at least an associated optimal trajectory $y_1(\cdot)$ such that $\nabla\phi (y_1(T))=0$ by Theorem \ref{ZeroTra}, but for any optimal trajectory $y_2(\cdot)$ associated to $p_2$ it holds that $\nabla\phi (y_2(T))\neq 0$ by Theorem \ref{Strano}.
\end{proof}
%If the hypothesis $\nabla \phi (x)\neq 0$ for all $x\in\mathbb{R}^n$ is not satisfied, we can prove the following result.
It might happen that two or more optimal trajectories actually start from a point $(t,x)$ at which $V$ is differentiable. However, if $H\in C^{1,1}_{loc}(\mathbb{R}^n \times (\mathbb{R}^n \smallsetminus \lbrace 0 \rbrace))$, then it is well-know that such a behaviour can only occur  when the gradient of $V$ at $(t,x)$ vanishes (see e.g. Theorem 7.3.14 and Example 7.2.10(iii) in \cite{MR2041617}).  More is true in one space dimension,  as we explain below.
%However, in our context, this may happen even at other points, since the uniqueness of solutions of \eqref{NewHamilt} is not guaranteed. 
\begin{example} 
In the one dimensional case it is easy to show that, if $V$ is differentiable at some point $(t_0,x_0)$ with  $V_x(t_0,x_0)\neq 0$, then there exists a unique optimal trajectory starting from $(t_0,x_0)$. Indeed, in this case, $F(x)=[f(x),g(x)]$ for suitable  functions  $f,g:\mathbb{R}\rightarrow \mathbb{R}$, with $f\leq g$, such that $-f$ and $g$ are locally semiconvex. So,
\begin{equation}\label{exmp}
H(x,p)=
\left\lbrace
\begin{array}{ll}
f(x)p,& p< 0,\\
g(x)p,& p\geq 0.
\end{array}
\right.
\end{equation}
If $x_0(\cdot)$ is an optimal trajectory at $(t_0,x_0)$ and $p_0(\cdot)$ is a dual arc associated with $x_0(\cdot)$, then by Remark~\ref{RelarkPartialSensRelation} we have that $0\neq V_x(t_0,x_0)=-p_0(t_0)$. Therefore, $0\neq p_0(t)$ for all $t\in [t_0,T]$ by  Remark \ref{RemarkDualArc}.  Thus,   \eqref{exmp} and the Maximum Principle yield
\begin{equation}\label{ODE}
\dot x_0(t)=
\left\lbrace
\begin{array}{ll}
f(x_0(t)),& \mbox{ if } V_x(t_0,x_0)>0,\\
g(x_0(t)),& \mbox{ if } V_x(t_0,x_0)<0.
\end{array}
\right.
\end{equation}
Since $f$ and $g$ are both locally Lipschitz, $x_0(\cdot)$ is the unique solution of \eqref{ODE} satisfying $x(t_0)=x_0$.
\end{example}

\section*{Acknowledgements}

Partial support of this research by the European Commission (FP7-PEOPLE-2010-ITN, Grant Agreement no.\ 264735-SADCO), and by the INdAM National Group GNAMPA is gratefully acknowledged.

%\cleardoublepage
%\phantomsection 
\addcontentsline{toc}{chapter}{\bibname}
% \defbibfilter{altro}{}
\printbibliography
\end{document}